\documentclass[11pt]{gtart}
\usepackage{amsmath, amssymb, graphicx, epsfig, graphs, verbatim, psfrag}
\usepackage{epic}

\newtheorem{thm}{Theorem}[section]

\newtheorem{lem}[thm]{Lemma}
\newtheorem{prop}[thm]{Proposition}

\theoremstyle{definition}

\newtheorem{rem}[thm]{Remark}

\numberwithin{equation}{section}

\newcommand{\al}{\alpha}

\newcommand{\ga}{\gamma}
\newcommand{\Ga}{\Gamma}

\newcommand{\si}{\sigma}

\newcommand{\ze}{\zeta}
\newcommand{\bfz}{{\mathbb {Z}}}
\newcommand{\bfq}{{\mathbb {Q}}}

\newcommand{\x}{\times}

\renewcommand{\t}{\mathbf t}

\newcommand{\TT}{\mathcal T}
\newcommand{\FF}{\mathcal F}
\newcommand{\LL}{\mathcal L}

\newcommand{\C}{\mathbb C}
\newcommand{\Z}{\mathbb Z}

\newcommand{\Q}{\mathbb Q}

\newcommand{\CP}{{\mathbb C}{\mathbb P}}

\newcommand{\del}{\partial}

\newcommand{\lra}{\longrightarrow\ }

\newcommand{\hf}{{{\widehat {HF}}}}

\DeclareMathOperator{\PD}{PD}
\DeclareMathOperator{\Spin}{Spin}

\begin{document}

\title{On the existence of tight contact structures on Seifert fibered 3--manifolds}

\author{Paolo Lisca}
\address{Dipartimento di Matematica ``L. Tonelli'',\\
Universit\`a di Pisa \\
Largo Bruno Pontecorvo 5\\
I-56127 Pisa, Italy} 
\email{lisca@dm.unipi.it}

\secondauthor{Andr\'{a}s I. Stipsicz}
\secondaddress{R\'enyi Institute of Mathematics\\
Re\'altanoda utca 13--15, Budapest H-1053, Hungary and\\
Mathematics Department, Columbia University\\
2990 Broadway, New York, NY 10027}
\secondemail{stipsicz@math-inst.hu}

\begin{abstract}
We determine the closed, oriented Seifert fibered 3--manifolds which
carry positive tight contact structures. Our main tool is a new
non--vanishing criterion for the contact Ozsv\'ath--Szab\'o invariant.
\end{abstract}



\maketitle

\section{Introduction}
\label{s:intro}

Let $Y$ denote a closed, oriented 3--manifold. A smooth 1--form $\alpha \in \Omega^1 (Y)$ is 
a~\emph{positive contact form} on $Y$ if $\alpha\wedge d\alpha >0$. The 2--plane field
\[
\xi := \ker \alpha\subset TY
\]
is called a \emph{(cooriented) positive contact structure}.  An embedded disk $D^2\subset Y$ is an
\emph{overtwisted disk} for $\xi$ if $TD^2=\xi$ along $\partial D^2$. The
contact structure $\xi $ is overtwisted if $Y$ contains an overtwisted
disk for $\xi$, otherwise $\xi$ is called~\emph{tight}.  

According to~\cite{Eli} every homotopy class of 2--plane fields on a
closed, oriented 3--manifold contains a unique up to isotopy
overtwisted contact structure, therefore the classification of
overtwisted structures reduces to a homotopy theoretic question.
Tight contact structures are much harder to find in general. In fact,
their existence is not known for a general 3--manifold $Y$, although
their presence seems to be related to the geometry of the underlying
3--manifold. Tight contact structures up to isotopy are classified on
$S^3$, lens spaces~\cite{Gir, Ho1}, circle bundles over
surfaces~\cite{Ho2} and some special Seifert fibered
3--manifolds~\cite{Ghiggold, Ghigg, GLS1, GLS2, Wu}.  In this paper we
address the existence question for tight contact structures on general
Seifert fibered 3--manifolds.

Using Legendrian surgery, Gompf \cite{Gompf} showed that each Seifert
fibered 3--manifold admits a Stein fillable (hence tight) positive
contact structure for at least one choice of orientation.  Gompf also
conjectured that the Poincar\'e homology sphere $-P$ with its
nonstandard orientation possesses no Stein fillable contact
structure. Gompf's conjecture was verified in~\cite{paolo} using
Seiberg--Witten theory, while later Etnyre and Honda~\cite{EH} showed
that $-P$ admits no positive, tight contact structures.  This result
was extended in~\cite{OSzII} to an infinite family $\{M_n\}_{n\geq 1}$
of small Seifert fibered 3--manifolds, described in the next
paragraph.

Let $T_{2,2n+1}\subset S^3$, for each integer $n\geq 1$, denote the
$(2,2n+1)$--torus knot, whose planar diagram is illustrated in
Figure~\ref{f:torusknot}.
\begin{figure}[ht]
\begin{center}
\psfrag{2n1}{\small $2n+1$ crossings}
\epsfig{file=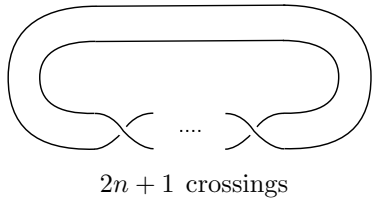, height=2.5cm}
\end{center}
\caption{The diagram of the torus knot $T_{2,2n+1}$, $n\geq 1$.}
\label{f:torusknot}
\end{figure}
Let $M_n$, for each $n\geq 1$, be the 3--manifold obtained by
performing $(2n-1)$--surgery along $T_{2,2n+1}$ in $S^3$. The
3--manifold $M_n$ can be also be viewed as the boundary of the
4--dimensional plumbing prescribed by the weighted tree of
Figure~\ref{f:Mnplumb}, where weights equal to $-2$ are omitted.  (For
the equivalence of the two presentations of $M_n$ see
\cite[Figure~2]{OSzII}.)
\begin{figure}[ht]
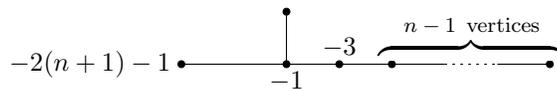

\begin{center}
\setlength{\unitlength}{1mm} \unitlength=0.7cm
\begin{graph}(6,1)(-2,0)
\graphnodesize{0.15}

\roundnode{n1}(-1,0)
\roundnode{n2}(1,0)
\roundnode{n3}(2,0)
\roundnode{n4}(3,0)
\rectnode{n5}[0,0](4,0)
\rectnode{n6}[0,0](5,0)
\roundnode{n7}(6,0)
\roundnode{n9}(1,1)

\edge{n1}{n2}
\edge{n2}{n3}
\edge{n3}{n4}
\edge{n4}{n5}
\edge{n5}{n6}[\graphlinedash{1 2}]
\edge{n6}{n7}
\edge{n2}{n9}

\autonodetext{n1}[w]{\small $-2(n+1)-1$} 
\autonodetext{n2}[s]{\small $-1$} 
\autonodetext{n3}[n]{\small $-3$}
\freetext(4.5,0.53){$\overset{\text{$n - 1$ vertices}}{\overbrace{\hspace{70pt}}}$} 

\end{graph}
\end{center}
\caption{The plumbing tree describing $M_n$}
\label{f:Mnplumb}
\end{figure}
It is well known that $M_n$ carries a Seifert fibered structure
for each $n\geq 1$, and the manifold $M_1$ is diffeomorphic to $-P$ above. 
The main result of the present paper is

\begin{thm}\label{t:main}
Let $Y$ be a closed, oriented Seifert fibered 3--manifold. Then, either $Y$ 
is orientation--preserving diffeomorphic to $M_n$ for some $n\geq 1$, or $Y$ 
carries a positive, tight contact structure.
\end{thm}

Since by~\cite[Corollary~1.2]{OSzII} each 3--manifold $M_n$ carries no
positive tight contact structures, Theorem~\ref{t:main} yields a
complete solution to the existence problem for positive tight contact
structures on Seifert fibered 3--manifolds.

The paper is organized as follows. In Section~\ref{s:seifert} we
collect the results about tight contact structures on Seifert fibered
3--manifolds known before Heegaard Floer theory.  In
Section~\ref{s:HF} we introduce Heegard Floer theory methods and we
use them to give a new criterion for the existence of tight contact
structures on Seifert fibered 3--manifolds. In
Sections~\ref{s:firstappl} and~\ref{s:secondappl} we apply the
criterion to prove the existence of tight contact structures for
several families of Seifert fibered 3--manifolds. In
Section~\ref{s:final} we use the results of
Sections~\ref{s:seifert},~\ref{s:firstappl} and~\ref{s:secondappl} to
prove Theorem~\ref{t:main}.

{\bf Acknowledgements:} The second author was partially supported by OTKA
49449, by EU Marie Curie TOK program BudAlgGeo and by Clay Mathematics
Institute.

\section{First constraints on the Seifert invariants}
\label{s:seifert}

In this section we collect several known results on the existence of tight contact 
structures on Seifert fibered 3--manifolds, summarizing in Proposition~\ref{p:reduction2} 
what was known before  Heegard Floer theory. For definitions and basic facts about Seifert 
fibered 3--manifolds we refer to~\cite{Orlik}.

Let $f\co Y_{e_0}\to S^2$ be an oriented three--dimensional circle
bundle over the 2--sphere, with Euler number $e_0\in\Z$. Let $F_1,
\ldots, F_k\subset Y_{e_0}$ be $k$ distinct fibers of the map $f$, and
denote by $Y(e_0; r_1,\ldots,r_k)$, with $r_i\in (0,1)\cap\Q$, the
oriented 3--manifold resulting from $(-\frac{1}{r_i})$--surgery along
$F_i$, $i=1,\ldots, k$, with the convention that the $0$--framing on
$F_i$ is given naturally by the fibration $f$.  It is a well--known
fact that each manifold $Y(e_0; r_1,\ldots,r_k)$ carries a Seifert
fibration over $S^2$ with $k$ multiple fibers and, conversely, each
oriented Seifert fibered 3--manifold with base $S^2$ and $k$ multiple
fibers is orientation--preserving diffeomorphic to
$Y(e_0;r_1,\ldots,r_k)$ for some $r_i\in (0,1)\cap\Q$, $i=1,\ldots, k$
and $e_0\in \Z$.

The~\emph{rational Euler characteristic} of $Y(e_0; r_1,\ldots,r_k)$
is, by definition,
\[
e(Y(e_0; r_1,\ldots,,r_k)):=e_0+r_1+\cdots+r_k. 
\] 
A simple computation shows that the 3--manifold $Y=Y(e_0;
r_1,\ldots,r_k)$ is a rational homology sphere, that is $b_1(Y)=0$, if
and only if $e(Y)\neq 0$.

\begin{prop}\label{p:reduction1}
Let $Y$ be a closed, oriented, Seifert fibered 3--manifold.  Then,
either $Y$ carries a tight contact structure or $Y$ is
orientation--preserving diffeomorphic to $Y(-1;r_1,r_2,r_3)$ for some
$r_i\in (0,1)\cap\Q$, $i=1,2,3$, with $e(Y(-1;r_1,r_2,r_3))>0$.
\end{prop}

\begin{proof} 
In~\cite[Theorem~5.4]{Gompf} the existence of Stein fillable (hence
tight) contact structures is proved for a Seifert fibered 3--manifold
$Y\to F$ provided either $F\neq S^2$ or $Y=Y(e_0;r_1,\ldots ,r_k)$
with $e_0\neq -1$.  Therefore, to prove the proposition it suffices to
argue that a Seifert fibered 3--manifold $Y$ orientation--preserving
diffeomorphic to $Y(-1;r_1,\ldots,r_k)$ for some $r_i\in (0,1)\cap\Q$,
$i=1,2,3$, carries a tight contact structure provided either $k\neq 3$
or $e(Y(-1;r_1,\ldots,r_k))\leq 0$.
If the number of multiple fibers $k\leq 2$ then $Y$ is a lens space,
which is well--known to carry tight contact structures~\cite{Gir}.  If
$k\geq 4$ then $Y$ contains incompressible tori and therefore it
admits infinitely many distinct tight structures by~\cite{CGH}.
If $e(Y)=0$ then $Y$ admits a smooth foliation $\FF$ transverse to the Seifert 
fibration~\cite{EHN}. Moreover, since the fibration has $3$ multiple fibers we have $Y\not\cong S^2\x S^1$. 
Therefore $\FF$ is a taut foliation and by~\cite[Theorem~2.4.1 and 
Corollary~3.2.8]{ET} it can be approximated by tight contact structures. Finally, if
$e(Y)<0$ then $Y$ is the link of an isolated surface singularity with $\C^*$--action~\cite[Corollary~5.3]{NR}, and as such it is known to carry tight contact structures 
(see e.g.~\cite{CNP}).
\end{proof}

Proposition~\ref{p:reduction2} below shows that for many of the
manifolds $Y(-1;r_1,r_2,r_3)$ with $e(Y)>0$ the conclusion of
Theorem~\ref{t:main} holds. In order to state the proposition we need
some preparation which will be useful also later on. Let
\[
Y=Y(-1; r_1, r_2, r_3),\quad r_1,r_2,r_3\in (0,1)\cap\Q,
\] 
be a small Seifert fibered 3--manifold with $e_0(Y)=-1$. From now on we will 
assume that $r_1\geq r_2\geq r_3$, and that there are continued fraction expansions
\begin{equation}\label{e:cfracs}
\frac 1{r_1} = [a_1, \ldots ,a_{n_1}],\quad
\frac 1{r_2} = [b_1, \ldots , b_{n_2}],\quad
\frac 1{r_3} = [c_1, \ldots , c_{n_3}],
\end{equation}
for some integers $a_i, b_j, c_k \geq 2$, where, by definition, 
\[
[x_1,\ldots,x_n]:=
x_1 - \cfrac{1}{x_2 -
          \cfrac{1}{\ddots -
             \cfrac{1}{x_n
}}}.
\]
Expansions~\eqref{e:cfracs} determine a plumbing tree $\Ga$ as in 
Figure~\ref{f:generaltree} and hence, as the result of the corresponding plumbing construction, 
an oriented 4--manifold $W_\Ga$ with $\del W_\Ga=Y$. It is not hard to show that $b_2^+(W_\Ga)=1$ if and only if $e(Y)>0$~\cite{NR}. As indicated in Figure~\ref{f:generaltree}, we will denote by $L_i$, for $i=1,2,3$, the leg of the weighted tree  $\Ga$ corresponding to $r_i$. More precisely, $L_1$ will denote the set of vertices of $\Ga$ with weights $-a_1,\ldots, -a_{n_1}$, 
and analogously for $L_2$ and $L_3$. Moreover, we will denote by $l(L_i)$, $i=1,2,3$, the length of $L_i$, that is its cardinality. 
\begin{figure}[ht]
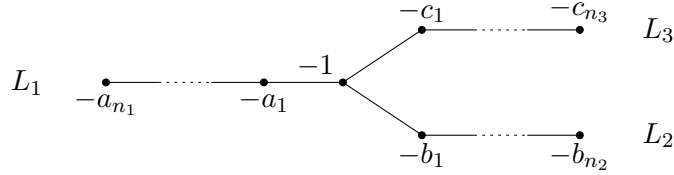

\begin{center}
\setlength{\unitlength}{1mm}
\unitlength=0.7cm
\begin{graph}(12,3.5)(-1,-1)
\graphnodesize{0.15}

\roundnode{m1}(1,1)
\rectnode{m2}[0,0](2,1)
\rectnode{m3}[0,0](3,1)
\roundnode{m4}(4,1)

\roundnode{m5}(5.5,1)

\roundnode{m6}(7,0)
\rectnode{m7}[0,0](8,0)
\rectnode{m8}[0,0](9,0)
\roundnode{m9}(10,0)

\roundnode{m10}(7,2)
\rectnode{m11}[0,0](8,2)
\rectnode{m12}[0,0](9,2)
\roundnode{m13}(10,2)

\edge{m1}{m2}
\edge{m2}{m3}[\graphlinedash{1 2}]
\edge{m3}{m4}
\edge{m4}{m5}
\edge{m5}{m6}
\edge{m6}{m7}
\edge{m7}{m8}[\graphlinedash{1 2}]
\edge{m8}{m9}

\edge{m5}{m10}
\edge{m10}{m11}
\edge{m11}{m12}[\graphlinedash{1 2}]
\edge{m12}{m13}

\autonodetext{m1}[s]{$-a_{n_1}$}
\autonodetext{m4}[s]{$-a_1$}
\autonodetext{m5}[nw]{$-1$}
\autonodetext{m6}[s]{$-b_1$}
\autonodetext{m9}[s]{$-b_{n_2}$}
\autonodetext{m10}[n]{$-c_1$}
\autonodetext{m13}[n]{$-c_{n_3}$}

\freetext(-0.5,1){$L_1$}
\freetext(11.5,0){$L_2$}
\freetext(11.5,2){$L_3$}

\end{graph}
\end{center}
\caption{\quad The plumbing tree $\Ga$ associated with $Y(-1;r_1,r_2,r_3)$}
\label{f:generaltree}
\end{figure}
Similarly, we have 
\[
-Y(-1;r_1,r_2,r_3)=Y(-2;1-r_1,1-r_2,1-r_3) = \del W_{\Ga'}, 
\]
where $\Ga'$ is the weighted tree ``dual'' to $\Ga$, determined by the
continued fraction expansions of $\frac 1{1-r_i}$, $i=1,2,3$. A useful
formulation of the relationship between the continued fraction
expansions of $\frac 1r$ and $\frac 1{1-r}$ is given by
Riemenschneider's point rule~\cite{Ri}. The dual tree $\Ga'$ is
illustrated in Figure~\ref{f:generaldualtree}.
\begin{figure}[ht]
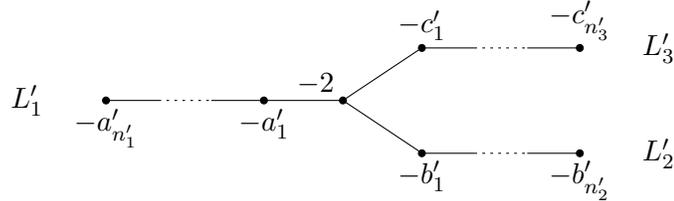

\begin{center}
\setlength{\unitlength}{1mm}
\unitlength=0.7cm
\begin{graph}(12,3.5)(-1,-1)
\graphnodesize{0.15}

\roundnode{m1}(1,1)
\rectnode{m2}[0,0](2,1)
\rectnode{m3}[0,0](3,1)
\roundnode{m4}(4,1)

\roundnode{m5}(5.5,1)

\roundnode{m6}(7,0)
\rectnode{m7}[0,0](8,0)
\rectnode{m8}[0,0](9,0)
\roundnode{m9}(10,0)

\roundnode{m10}(7,2)
\rectnode{m11}[0,0](8,2)
\rectnode{m12}[0,0](9,2)
\roundnode{m13}(10,2)

\edge{m1}{m2}
\edge{m2}{m3}[\graphlinedash{1 2}]
\edge{m3}{m4}
\edge{m4}{m5}
\edge{m5}{m6}
\edge{m6}{m7}
\edge{m7}{m8}[\graphlinedash{1 2}]
\edge{m8}{m9}

\edge{m5}{m10}
\edge{m10}{m11}
\edge{m11}{m12}[\graphlinedash{1 2}]
\edge{m12}{m13}

\autonodetext{m1}[s]{$-a'_{n'_1}$}
\autonodetext{m4}[s]{$-a'_1$}
\autonodetext{m5}[nw]{$-2$}
\autonodetext{m6}[s]{$-b'_1$}
\autonodetext{m9}[s]{$-b'_{n'_2}$}
\autonodetext{m10}[n]{$-c'_1$}
\autonodetext{m13}[n]{$-c'_{n'_3}$}

\freetext(-0.5,1){$L'_1$}
\freetext(11.5,0){$L'_2$}
\freetext(11.5,2){$L'_3$}

\end{graph}
\end{center}
\caption{\quad The ``dual'' tree $\Ga'$ associated with $-Y(-1;r_1,r_2,r_3)$}
\label{f:generaldualtree}
\end{figure}

The proof of Proposition~\ref{p:reduction2} requires the use of contact 
surgery~\cite{DG1, DG2, DGS}, so we briefly recall the necessary notions.
Suppose that $L\subset (Y, \xi )$ is a Legendrian knot in a
contact 3--manifold.  The contact structure equips $L$ with a framing
(that is, a trivialization of its normal bundle) called the~\emph{contact
framing} of $L$. Let $Y_L^{\pm }$ denote the 3--manifold obtained by 
$(\pm 1)$--surgery along $L$, where the surgery coefficient is
measured with respect to the contact framing of $L$. According to the
classification of tight contact structures on a solid
torus~\cite{Ho1}, the restriction of $\xi$ to the complement of a standard 
neighborhood of $L$ extends
uniquely, up to isotopy, to the surgered manifolds $Y_L^+$ and $Y_L^-$, 
restricting as a tight structure on the glued--up torus. Therefore, the knot $L$ decorated 
with a $(+1)$ or $(-1)$ uniquely specifies a contact 3--manifold
$(Y_L^+, \xi _L^+)$ or $(Y_L^-, \xi _L^-)$. By~\cite{Eli, Gompf} any contact $(-1)$--surgery 
along a link in the standard contact 3--sphere produces a Stein fillable, hence tight 
contact structure. The notion of contact $(\pm 1)$--surgery can be extended to any
nonzero rational surgery along a Legendrian knot. The extension of the contact 
structure is unique, however, only for surgery coefficients of the form
$\frac{1}{k}$, with $k\in \bfz $. In~\cite{DG2, DGS} it is shown that
a rational contact surgery can be replaced by a sequence of contact
$(\pm 1)$--surgeries. For negative surgeries only $(-1)$--surgeries 
are needed in the replacement.

\begin{prop}\label{p:reduction2}
Let $Y$ be an oriented, Seifert fibered 3--manifold which is not 
orientation--preserving diffeomorphic to $Y(-1;r_1,r_2,r_3)$, with  $1>r_1\geq r_2\geq r_3>0$
satisfying~\eqref{e:cfracs} and each of the following: 
\begin{itemize}
\item
$e(Y(-1;r_1,r_2,r_3)) = -1 + r_1 + r_2 + r_3 > 0$;
\item
$a_1=\cdots=a_k=2$ for some $k\geq 1$ and either 
\begin{itemize}
\item
$n_1=k$ or 
\item
$n_1>k$ and $a_{k+1}>2$;
\end{itemize}
\item
$c_1\geq b_1=k+2$. 
\end{itemize}
Then, $Y$ carries a tight contact structure.
\end{prop}

\begin{proof}
By Proposition~\ref{p:reduction1} we may assume $Y\cong Y(-1;r_1,r_2,r_3)$, with 
$e(Y)>0$ for some $1>r_1\geq r_2\geq r_3>0$. 
Notice that $r_1\geq r_2\geq r_3$ implies $a_1\leq
b_1\leq c_1$. We can think of the weighted tree $\Ga$ of Figure~\ref{f:generaltree} 
as prescribing an integral surgery diagram for $Y$, with each vertex corresponding 
to an unknot, and each weight corresponding to a surgery coefficient. 
In the case  $a_1\geq 3$ we can ``blow down'', in the sense of Kirby calculus, 
the central $(-1)$--circle to get a surgery diagram of unknots, each 
with surgery coefficient $\leq -2$. Moreover, it is easy to see that the 
resulting framed link can be isotoped to Legendrian position so that on each component 
the required topological surgery can be realized by some negative contact surgery. 
Therefore, by well--known results~\cite{Gompf} in this case $Y$ 
carries Stein fillable structures. This means that we may assume $a_1=\cdots =a_k=2$ for some 
$k\geq 1$, so either $n_1=k$ or $n_1>k$ and $a_{k+1}>2$. 

If $b_1\geq k+3$ we can blow down  the $k+1$  unknots with framing $(-1)$, 
corresponding to the central vertex of $\Ga$ together with the first $k$ vertices of $L_1$. 
The components of the resulting framed link $\LL$ are pairwise positively 
linked, and it is easy to see as before that $\LL$ has a Legendrian 
representative such that each topological surgery can be realized by a negative 
contact surgery. By~\cite{DG2, DGS} and~\cite{Gompf} this implies that the 3--manifold $Y$ resulting from the surgery carries Stein fillable structures, so we may assume $b_1\leq k+2$. 

To conclude the proof it suffices to show that $Y$ carries Stein fillable 
structures if $b_1\leq k+1$. In this case we blow--down $b_1$  
 $(-1)$--circles instead of the available $k+1$. After the blow--down operations, 
the unknot $U$ corresponding to the first vertex of $L_2$ has framing 
$0$, and the unknot $U'$ corresponding to the first vertex of $L_3$ has 
framing $-c_1+b_1\leq 0$. Moreover, since $b_1\geq 2$, $U$ and $U'$ link positively 
at least twice. The result of the $0$--surgery on $U$ can be viewed as 
$S^2\x S^1$.  Then, due to the linking between $U$ and $U'$, the remaining topological 
surgeries in $S^2\x S^1$ can be realized by negative contact surgeries on 
$(S^2\x S^1,\xi_0)$, where $\xi_0$ is the standard Stein fillable contact structure 
on $S^2\x S^1$ (see e.g.~\cite[Section~3]{GLS1} for similar arguments).
\end{proof}

\begin{rem}\label{r:delicate}
Observe that each of the 3--manifolds $M_n$ defined in Section~\ref{s:intro}, known not to admit
tight contact structures, falls outside the range of applicability of Proposition~\ref{p:reduction2}. 
Thus, we can rephrase Proposition~\ref{p:reduction2} by saying that in order to prove Theorem~\ref{t:main} 
it suffices to establish the existence of tight contact structures on each 3--manifold distinct from every
$M_n$ and associated with a plumbing tree $\Ga$ as in Figure~\ref{f:specialtree}. 
The proof of this existence result  will occupy the rest of the paper. We will need arguments of a fairly delicate nature when compared with those used in the proof of Proposition~\ref{p:reduction2}. The dual tree $\Ga'$ is shown in Figure~\ref{f:specialdualtree}, where 
weights $-2$ are omitted. Moreover: 
\[
a'_1=
\begin{cases} 
k+1\quad\text{if}\quad n_1=k,\\
k+2\quad\text{if}\quad n_1>k,
\end{cases}
\quad\text{and}\quad
b'_{k+1}
\begin{cases}
=2\quad\text{if}\qua n_2=1,\\
>2\quad\text{if}\quad n_2>1.

\end{cases}
\]
\end{rem}

\begin{figure}[ht]
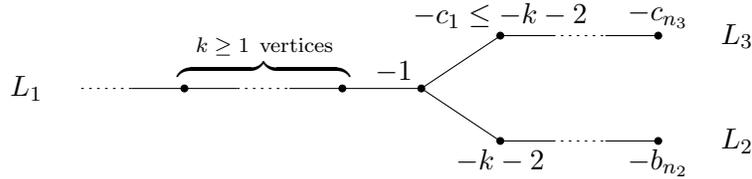

\begin{center}
\setlength{\unitlength}{1mm}
\unitlength=0.7cm
\begin{graph}(13,3.5)(-2,-1)
\graphnodesize{0.15}

\rectnode{m-1}[0,0](-1,1)
\rectnode{m0}[0,0](0,1)
\roundnode{m1}(1,1)
\rectnode{m2}[0,0](2,1)
\rectnode{m3}[0,0](3,1)
\roundnode{m4}(4,1)

\roundnode{m5}(5.5,1)

\roundnode{m6}(7,0)
\rectnode{m7}[0,0](8,0)
\rectnode{m8}[0,0](9,0)
\roundnode{m9}(10,0)

\roundnode{m10}(7,2)
\rectnode{m11}[0,0](8,2)
\rectnode{m12}[0,0](9,2)
\roundnode{m13}(10,2)

\edge{m0}{m-1}[\graphlinedash{1 2}]
\edge{m0}{m1}
\edge{m1}{m2}
\edge{m2}{m3}[\graphlinedash{1 2}]
\edge{m3}{m4}
\edge{m4}{m5}
\edge{m5}{m6}
\edge{m6}{m7}
\edge{m7}{m8}[\graphlinedash{1 2}]
\edge{m8}{m9}

\edge{m5}{m10}
\edge{m10}{m11}
\edge{m11}{m12}[\graphlinedash{1 2}]
\edge{m12}{m13}

\autonodetext{m5}[nw]{$-1$}
\autonodetext{m6}[s]{$-k-2$}
\autonodetext{m9}[s]{$-b_{n_2}$}
\autonodetext{m10}[n]{$-c_1\leq -k-2$}
\autonodetext{m13}[n]{$-c_{n_3}$}

\freetext(-2,1){$L_1$}
\freetext(11.5,0){$L_2$}
\freetext(11.5,2){$L_3$}
\freetext(2.5,1.6){$\overset{\text{$k\geq 1$ vertices}}{\overbrace{\hspace{65pt}}}$}

\end{graph}
\end{center}
\caption{The constrained plumbing tree $\Ga$}
\label{f:specialtree}
\end{figure}

\begin{figure}[ht]
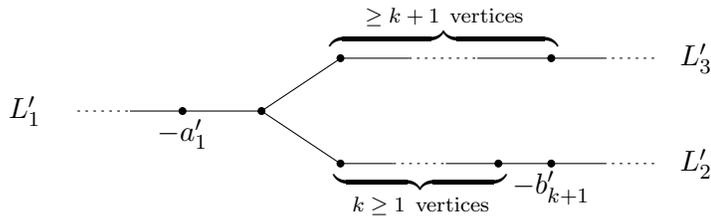

\begin{center}
\setlength{\unitlength}{1mm}
\unitlength=0.7cm
\begin{graph}(14,3.5)(0,-1)
\graphnodesize{0.15}

\rectnode{m2}[0,0](2,1)
\rectnode{m3}[0,0](3,1)
\roundnode{m4}(4,1)

\roundnode{m5}(5.5,1)

\roundnode{m6}(7,0)
\rectnode{m7}[0,0](8,0)
\rectnode{m8}[0,0](9,0)
\roundnode{m9}(10,0)
\roundnode{m10}(11,0)
\rectnode{m11}[0,0](12,0)
\rectnode{m12}[0,0](13,0)

\roundnode{m14}(7,2)
\rectnode{m15}[0,0](8.3,2)
\rectnode{m16}[0,0](9.6,2)
\roundnode{m17}(11,2)
\rectnode{m18}[0,0](12,2)
\rectnode{m19}[0,0](13,2)

\edge{m2}{m3}[\graphlinedash{1 2}]
\edge{m3}{m4}
\edge{m4}{m5}
\edge{m5}{m6}
\edge{m6}{m7}
\edge{m7}{m8}[\graphlinedash{1 2}]
\edge{m8}{m9}
\edge{m9}{m10}
\edge{m10}{m11}
\edge{m11}{m12}[\graphlinedash{1 2}]

\edge{m5}{m14}
\edge{m14}{m15}
\edge{m15}{m16}[\graphlinedash{1 2}]
\edge{m16}{m17}
\edge{m17}{m18}
\edge{m18}{m19}[\graphlinedash{1 2}]

\autonodetext{m4}[s]{$-a'_1$}
\autonodetext{m10}[s]{$-b'_{k+1}$}

\freetext(8.5,-0.6){$\underset{\text{$k\geq 1$ vertices}}{\underbrace{\hspace{65pt}}}$}
\freetext(8.95,2.6){$\overset{\text{$\geq k+1$ vertices}}{\overbrace{\hspace{86pt}}}$}

\freetext(1,1){$L'_1$}
\freetext(13.75,0){$L'_2$}
\freetext(13.75,2){$L'_3$}

\end{graph}
\end{center}
\caption{\quad The constrained dual tree $\Ga'$}
\label{f:specialdualtree}
\end{figure}

\section{Contact invariants and tight contact structures}
\label{s:HF}

In this section we introduce and show how to use the crucial
ingredient in the proof of Theorem~\ref{t:main}: the contact
Ozsv\'ath--Szab\'o invariant. We first recall the basic facts of
Heegaard Floer theory and a result from~\cite{OSzIII}, which gives a
non--vanishing criterion for the contact Ozsv\'ath--Szab\'o
invariant. Then, after some preparatory material, we state and prove
Theorem~\ref{t:magic}, which gives a new method to apply the
non--vanishing criterion.

Heegaard Floer theory~\cite{OSzF1, OSzF2, OSzF4, OSz6} associates a finitely generated 
abelian group $\hf (Y, \t)$, the \emph{Ozsv\'ath--Szab\'o homology group}, to a closed, oriented spin$^c$ $3$--manifold $(Y, \t)$. Throughout this paper we will assume that $\Z/2\Z$ coefficients are being used in the complexes defining the $\widehat{HF}$--groups. With this assumption, the groups are actually $\Z/2\Z$--vector spaces.  The symbol $\hf (Y)$ will denote the direct sum of $\hf (Y, \t )$ for all spin$^c$
structures. A fundamental property of these groups is that on each 3--manifold there are
only finitely many spin$^c$ structures  with
non--trivial Ozsv\'ath--Szab\'o homology group, hence $\hf (Y)$ is also
finitely generated.  By \cite[Proposition~5.1]{OSzF2} a rational
homology sphere $Y$ has non--trivial Ozsv\'ath--Szab\'o homology group
$\hf (Y, \t )$ for each spin$^c$ structure $\t\in Spin ^c (Y)$.  In
particular, for a rational homology 3--sphere $Y$ we have
\[
\dim \hf (Y)\geq  \vert Spin^c (Y)\vert = \vert H_1(Y; \Z )\vert.  
\]
A rational homology 3--sphere $Y$ is called an~\emph{$L$--space} if
\[
\dim \hf (Y)= \vert H_1(Y; \Z )\vert.
\]
In view of the above nonvanishing result, this property is equivalent to 
\[
\hf (Y ,\t )=\Z /2\Z\quad\text{for each $\t\in\Spin^c(Y)$}.
\]
Recall that a cooriented contact structure $\xi$ on an oriented
3--manifold $Y$ determines a spin$^c$ structure $\t _{\xi }$ on $Y$
and, viewing $\xi$ as an oriented 2--plane bundle we have
$c_1(\xi)=c_1(\t_\xi)$. In~\cite{OSz6} Ozsv\'ath and Szab\'o define an
invariant, the~\emph{contact Ozsv\'ath--Szab\'o invariant}
\[
c(Y, \xi)\in \hf (-Y,\t_{\xi })
\]
assigned to a positive, cooriented contact structure $\xi$ on $Y$. A
basic property of this invariant is that if $(Y, \xi )$ is overtwisted
then $c(Y, \xi )=0$, and if $(Y, \xi )$ is Stein fillable then $c(Y, \xi
)\neq 0$.  In particular, for the standard contact structure $(S^3,
\xi _{st})$ the invariant $c(S^3, \xi _{st})\in \hf (S^3)=\Z /2\Z$ is
non--zero. Moreover, if $(Y_2,\xi _2)$ is given as Legendrian surgery
along a Legendrian knot in $(Y_1, \xi _1)$ and $c(Y_1, \xi _1)\neq 0$
then $c(Y_2, \xi _2)\neq 0$; in particular, $(Y_2, \xi _2)$ is
tight~\cite{LS3, OSz6}.

It was proved in~\cite{OSzF1} that if $\t \in Spin ^c(Y)$ and $c_1(\t
)\in H^2 (Y; \Z )$ is a torsion element then the Ozsv\'ath--Szab\'o
homology group $\hf (Y, \t )$ comes with a natural relative
$\Z$--grading. Moreover, this relative $\Z$--grading admits a natural
lift to an absolute $\Q$--grading~\cite{OSzabs}.  Thus, when $c_1(\t)$
is torsion the Ozsv\'ath--Szab\'o homology group $\hf (Y, \t)$ splits
as
\[
\hf (Y, \t )=\oplus _{n \in  \Z}\hf _{d_0+n} (Y, \t), 
\]
where the degree $d_0\in\Q$ is determined mod 1 by $\t$. To a rational
homology 3--sphere $Y$ endowed with a spin$^c$ structure $\t$ an
invariant $d(Y, \t)\in\Q$ is associated, called the~\emph{correction
term}~\cite{OSzabs}, which satisfies $d(-Y,\t)=-d(Y,\t)$ and
\[
\hf _{d(Y,\t)} (Y, \t)\neq 0.
\]
If $Y$ is an $L$--space $Y$ then $\hf (Y, \t)=\bfz _2$, therefore in this simple case  $d(Y, \t)$ is characterized
as the unique degree of a nontrivial element in $\hf (Y, \t)$. 

Given a cooriented 2--plane field $\xi$ on the oriented 3--manifold $Y$,  if $c_1(\xi )\in H^2 (Y; \Z )$ 
is torsion, $(X,J)$ is an almost complex 4--manifold such that 
$\del X =Y$ and $\xi$ is equal to the distribution of complex tangent lines to $\del X$, the rational number 
\[
d_3(\xi):= \frac{1}{4}(c_1^2(X, J)- 3\sigma (X) -2 b_2(X))\in\Q
\]
depends only on $\xi$ up to homotopy, and not on the choice of
$(X,J)$, see~\cite{Gompf}.  The degree of the contact invariant $c(Y,
\xi )\in \hf (-Y, \t _{\xi })$ is known to be equal to $-d_3(\xi )$.
Consequently, when $Y$ is an $L$--space it easily follows from
$c(Y,\xi)\neq 0$ that $d_3(\xi)=d(Y,\t_\xi)$.  In some cases the
converse also holds:

\begin{thm}[\cite{OSzIII}, Theorem~1.2]\label{t:char}
Let $Y=Y(-1;r_1,r_2,r_3)$, with $e(Y)>0$. Let $\xi$ be a contact structure on $Y$ 
given by a contact surgery diagrams as in Figure~\ref{f:OSzIII2}. Then, if $d_3(\xi )=d(Y,\t_\xi)$ 
we have $c(Y, \xi )\neq 0$ and, in particular, $\xi$ is tight.\qed
\end{thm}
\begin{figure}[ht]
\begin{center}
\psfrag{r1}{$-\frac1{r_1}$}
\psfrag{r2}{$-\frac1{r_2}$}
\psfrag{r3}{$-\frac1{r_3}$}
\psfrag{+1}{$+1$}
\epsfig{file=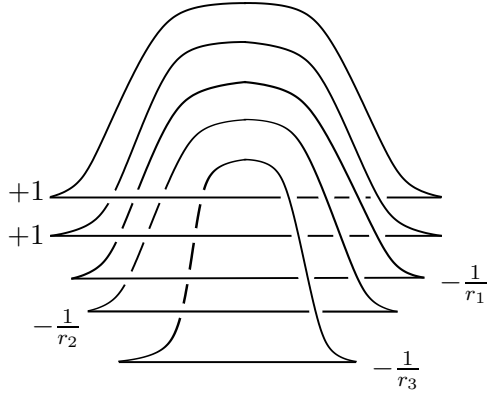, height=5cm}
\end{center}
\caption{Contact structures on $Y(-1;r_1,r_2,r_3)$}
\label{f:OSzIII2}
\end{figure}
To understand the statement of Theorem~\ref{t:char} it is important to
keep in mind that, since the rational numbers $-\frac 1{r_i}$ are not
necessarily of the form $\frac 1k$, $k\in\Z$, the contact surgeries
they determine are not unique (see
Section~\ref{s:seifert}). Therefore, for fixed $r_1$, $r_2$ and $r_3$,
Figure~\ref{f:OSzIII2} defines a finite collection of contact
structures on the same underlying topological 3--manifold
$Y(-1;r_1,r_2,r_3)$. See~\cite{OSzII, OSzIII} for more details and
explicit examples.

In~\cite{OSzplum} Ozsv\'ath and Szab\'o prove the existence of an
algorithm which computes $\hf(Y)$ assuming that $Y$ is the boundary of
a negative definite plumbing of a certain type. In
Sections~\ref{s:firstappl} and~\ref{s:secondappl} we will use the
Ozsv\'ath--Szab\'o algorithm to apply Theorem~\ref{t:char}. In order
to state our next result we need to recall the main ingredient of the
algorithm, that is the definition of~\emph{full path}.

Suppose that $\TT$ is a plumbing tree of spheres, $W_\TT$ is the
corresponding 4--manifold and $Y_\TT = \del W_\TT$. A vertex $v$ of
$\TT$ is~\emph{bad} if its valency is larger than the absolute value
of its weight. The algorithm exists assuming that $\TT$ is negative
definite and has at most one bad vertex. Observe that such assumptions
are satisfied if $\TT$ is equal to a weighted tree $\Ga'$ as in
Figure~\ref{f:specialdualtree}. Fix an identification of the set of
vertices of $\TT$ with a set of standard generators of
$H_2(W_\TT;\Z)$, so that $\TT$ coincides with the corresponding
intersection graph. An $m$--tuple $(K_1, \ldots ,K_m)$ of second
cohomology elements on $W_\TT$ is said to be a~\emph{full path} if
\begin{itemize}
\item 
each $K_i\in H^2 (W_\TT; \bfz )$ is a characteristic element, that is, 
\[
\langle K_i, v\rangle \equiv v\cdot v \quad \bmod 2\quad\text{for
every $v\in\TT$};
\]
\item 
$K_1$ is an \emph{initial vector}, that is, it satisfies
\[
v\cdot v + 2\leq \langle K_1, v\rangle\leq -v\cdot v\quad\text{for every $v\in\TT$};
\]
\item 
For $i=1, \ldots , m-1$ the vector $K_{i+1}$ is given by 
\[
K_{i+1} = K_i+2PD(v)
\]
for some $v\in\TT$ satisfying $\langle K_i, v\rangle = -v\cdot v$;
\item 
$K_m$ is a \emph{terminal vector}, that is
\[
v\cdot v\leq \langle K_m, v\rangle\leq -v\cdot v - 2 \quad\text{for every $v\in\TT$}.
\]
\end{itemize}
Notice that the length $m$ of a full path might vary. For example, if
$v\cdot v < -1$ for every $v$ then there is a vector $K$ which is both initial and terminal:
\[
\langle K, v\rangle := 
\begin{cases}
0\quad\text{if $v\cdot v$ is even},\\
1 \quad\text{if $v\cdot v$ is odd}.
\end{cases}
\]
Therefore in this case there is always a full path with $m=1$.
According to~\cite{OSzplum}, a full path $(K_1, \ldots , K_m)$ determines a non--trivial element of 
$\hf(Y_\TT)$ whose absolute degree can be computed using the formula
\begin{equation}\label{e:degree-formula}
\frac{1}{4}(K_i^2 +\vert \TT\vert ),
\end{equation}
where $K_i$ is any element of the full path. Notice that the
transformation rule defining $K_{i+1}$ from $K_i$ implies that
$K_i^2=K_{i+1}^2$, therefore the choice of the vector in the full path is
irrelevant when computing the degree. 

The construction given in the following lemma will be crucial in Sections~\ref{s:firstappl} and~\ref{s:secondappl}.

\begin{lem}\label{l:ratsur}
Let $Y=Y(-1;r_1,r_2,r_3)=\del W_\Ga$ with $\Ga$ as in
Figure~\ref{f:generaltree}.  Then, there exist a smooth, closed and
oriented 4--manifold $R$ containing $Y$ as a smooth hypersurface and
an open tubular neighborhood $\nu(Y)\subset R$ such that
\[
R\setminus\nu(Y)=W_\Ga \cup W_{\Ga'}.
\] 
Moreover, $R$ is orientation preserving diffeomorphic to a blown--up
complex projective plane.
\end{lem}

\begin{proof}
Start by blowing up the complex projective plane $\CP^2$ at the common
intersection point $p$ of three distinct lines $\ell_1$, $\ell_2$ and
$\ell_3$. The union of the exceptional class, the proper transforms
$\tilde\ell_i$ of the $\ell_i$, $i=1,2,3$, and the proper transform
$\tilde\ell$ of a line $\ell\subset\CP^2$ with $p\not\in\ell$ provides
a configuration of five rational curves inside the rational surface
$\CP^2\#\overline\CP^2$. We can now blow up more times, starting at
the three intersection points $\tilde\ell \cap\tilde\ell_i$,
$i=1,2,3$, until we obtain a configuration of rational curves in the
blown--up projective plane with intersection graph identical to
$\Ga$. This gives an embedding of the plumbing $W_\Ga$ into a blown-up
$\CP^2$, which we take as our $R$. Then it is easy to check that
$Y=\del W_\Ga$ embeds as a hypersurface in $R$ and it has an open
tubular neighborhood $\nu(Y)$ such that $R\setminus\nu(Y) = W_\Ga\cup
W_{\Ga'}$.
\end{proof} 

Recall that, given a contact structure $\xi$ on a Seifert fibered 3--manifold $Y$, 
$\xi$ is called a~\emph{transverse contact structure} if it can be isotoped until it is transverse 
everywhere to the fibers of the Seifert fibration on $Y$. Since transverse contact structures are symplectically fillable~\cite{LiMa} and therefore tight, the existence question for Seifert 
fibered 3--manifolds is only open in the absence of transverse contact structures. 
The following statement gives a practical criterion for the existence of tight contact structures on Seifert fibered 3--manifolds which do not carry transverse contact structures. It will be applied in Sections~\ref{s:firstappl} and~\ref{s:secondappl}.

\begin{thm}\label{t:magic}
Let $Y=Y(-1;r_1,r_2,r_3)=\del W_\Ga$ with $\Ga$ as in Figure~\ref{f:generaltree} and $e(Y)>0$. 
Suppose that $Y$ does not carry transverse contact structures, and let $\xi$ be a contact structure 
on $Y$ given by a surgery diagram as in Figure~\ref{f:OSzIII2}. 
Let $R$ be a smooth, closed 4--manifold containing $Y$ as in Lemma~\ref{l:ratsur}, and  
let $c\in H^2 (R; \bfz)$ be a characteristic cohomology class such 
that:
\begin{itemize}
\item
$d_3(\xi) = \frac{1}{4}((c|_{W_\Ga})^2-3\si (W_\Ga)-2 b_2 (W_\Ga))+1$;
\item
$c\vert _{W_{\Ga'}}$ belongs to a full path on $W_{\Ga'}$;
\item 
$c^2=\si(R)$.
\end{itemize}
Then, $c(Y,\xi)\neq 0$ and hence $\xi$ is tight.
\end{thm}

\begin{proof}
By Theorem~\ref{t:char} it suffices to check the equality $d_3(\xi
)=d(Y,\t_{\xi })$.  Since $e(Y)>0$, $Y$ is a rational homology
sphere. Therefore we have the splitting
\[
H^2(R; \bfq )\cong H^2 (W_\Ga; \bfq )\oplus H^2 (W_{\Ga'}; \bfq )\ni
c_\Ga+c_{\Ga'}=c,
\]
where $c_\Ga$ and $c_{\Ga'}$ abbreviates $c\vert _{W_{\Ga}}$ and
$c\vert _{W_{\Ga'}}$, respectively.  By a simple computation using the
fact that $b_2^+(W_\Ga)=1$ and the first assumption on $c$ we have

\begin{equation}\label{e:d1}
d_3(\xi )=\frac{1}{4}(c_{\Ga}^2-3\sigma (W_\Ga)-
2 b_2 (W_\Ga))+1=\frac{1}{4}(c^2_\Ga-\sigma (W_\Ga)).
\end{equation}
By the second assumption on $c$, the restriction $c _{\Ga'}$ determines a non--trivial element 
of $\hf(-Y,\t_\xi)$ which, by~\eqref{e:degree-formula}, has degree
\[
\frac{1}{4}(c_{\Ga'}^2+|\Ga'|) = \frac{1}{4}(c_{\Ga'}^2+b_2(W_{\Ga'})).
\]
Now observe that by assumption $Y$ has no transverse contact structures while,  
since $e(-Y)=-e(Y)<0$, $-Y$ does carry contact structures transverse to the Seifert fibration~\cite[Proposition~3.1]{LiMa}. Applying~\cite[Theorem~1.1]{OSzIII} 
gives that $Y$ is an $L$--space, therefore
\begin{equation}\label{e:d2}
d(-Y, \t_\xi)=\frac{1}{4}(c_{\Ga'}^2+b_2(W_{\Ga'}))=
\frac{1}{4}(c^2_{\Ga'}-\sigma (W_{\Ga'})).
\end{equation}
Adding Equations~\eqref{e:d1} and~\eqref{e:d2} and using the third assumption we get
\begin{equation}\label{e:d3}
d_3(\xi )+d(-Y, \t_\xi)=\frac{1}{4}(c^2-\si (R))=0.
\end{equation}
Since $d(Y,\t _\xi)=-d(-Y, \t_\xi)$, Identity~\eqref{e:d3} implies $d_3(\xi)=d(Y, \t_\xi)$, and 
the tightness of $\xi$ follows applying Theorem~\ref{t:char}.
\end{proof}

\section{First application of the criterion}
\label{s:firstappl}

In this section we apply Theorem~\ref{t:magic} to prove the following statement:

\begin{thm}\label{t:n1>1}
Let $Y = Y(-1;r_1,r_2,r_3)$ with $e(Y) > 0$, and suppose that $Y$ carries no transverse 
contact structures. Suppose that $1>r_1\geq r_2\geq r_3>0$ satisfy Expansions~\eqref{e:cfracs},  
and each of the following holds:
\begin{itemize}
\item
$a_1=\cdots=a_k=2$ and either 
\begin{itemize}
\item
$n_1=k$ or 
\item
$n_1>k$ and $a_{k+1}>2$;
\end{itemize}
\item
$c_1\geq b_1=k+2$;
\item
$n_1>1$;
\item
$n_3=1$.
\end{itemize}
Then, $Y$ carries a contact structure $\xi$ given by a surgery diagram as in Figure~\ref{f:OSzIII2} and such that $c(Y,\xi)\neq 0$. In particular, $Y$ carries a tight contact structure. 
\end{thm}

\begin{figure}[ht]
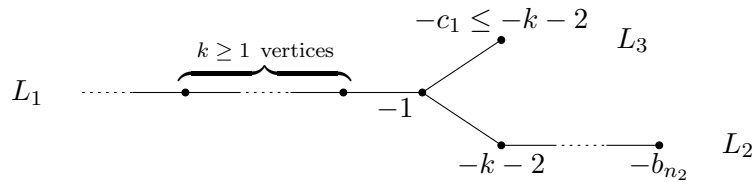

\begin{center}
\setlength{\unitlength}{1mm}
\unitlength=0.7cm
\begin{graph}(13,3.5)(-2,-1)
\graphnodesize{0.15}

\rectnode{m-1}[0,0](-1,1)
\rectnode{m0}[0,0](0,1)
\roundnode{m1}(1,1)
\rectnode{m2}[0,0](2,1)
\rectnode{m3}[0,0](3,1)
\roundnode{m4}(4,1)

\roundnode{m5}(5.5,1)

\roundnode{m6}(7,0)
\rectnode{m7}[0,0](8,0)
\rectnode{m8}[0,0](9,0)
\roundnode{m9}(10,0)

\roundnode{m10}(7,2)

\edge{m0}{m-1}[\graphlinedash{1 2}]
\edge{m0}{m1}
\edge{m1}{m2}
\edge{m2}{m3}[\graphlinedash{1 2}]
\edge{m3}{m4}
\edge{m4}{m5}
\edge{m5}{m6}
\edge{m6}{m7}
\edge{m7}{m8}[\graphlinedash{1 2}]
\edge{m8}{m9}

\edge{m5}{m10}

\autonodetext{m5}[sw]{$-1$}
\autonodetext{m6}[s]{$-k-2$}
\autonodetext{m9}[s]{$-b_{n_2}$}
\autonodetext{m10}[n]{$-c_1\leq -k-2$}

\freetext(2.5,1.6){$\overset{\text{$k\geq 1$ vertices}}{\overbrace{\hspace{65pt}}}$}
\freetext(-2,1){$L_1$}
\freetext(11.5,0){$L_2$}
\freetext(9.5,2){$L_3$}

\end{graph}
\end{center}
\caption{The tree $\Ga$ under the assumptions of Theorem~\ref{t:n1>1}}
\label{f:gr}
\end{figure}

Let $Y = \del W_\Ga$ be a 3--manifold as in the statement of Theorem~\ref{t:n1>1}. 
The corresponding tree $\Ga$ is illustrated in Figure~\ref{f:gr}, with $(-2)$--weights omitted. 
Therefore we have $-Y=\del W_{\Ga'}$, where the dual plumbing tree $\Ga'$ given in Figure~\ref{f:dugr}, 
with $(-2)$--weights omitted. Explicitely, we have:
\begin{itemize}
\item
The central vertex of $\Ga'$ has weigth $-2$.
\item
The leg $L_1'$ starts out with a vertex $w$ of weight $-e$. Moreover, 
\[
e=
\begin{cases}
k+1\ \text{if}\ l(L'_1) =1,\\
k+2\ \text{if}\ l(L'_1)>1.
\end{cases}
\]
\item
The leg $L_2'$ starts out with $k\geq 1$ vertices of weight $-2$, then has a
vertex $v$ of weight $-b\leq -2$ and then possibly more vertices. 
\item
The leg $L_3'$ has length $l(L_3')=c_1-1\geq k+1$, and each one of its vertices has
weight $-2$.
\end{itemize}
\begin{figure}[ht]
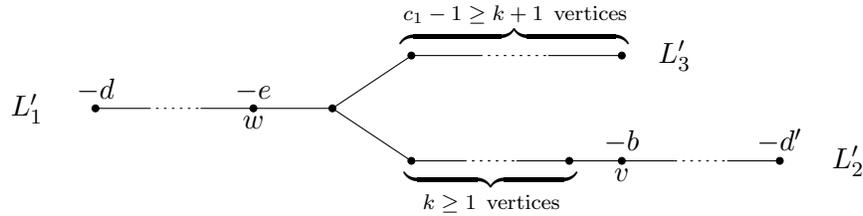

\begin{center}
\setlength{\unitlength}{1mm}
\unitlength=0.7cm
\begin{graph}(15,3.5)(0,-1)
\graphnodesize{0.15}

\roundnode{m1}(1,1)
\rectnode{m2}[0,0](2,1)
\rectnode{m3}[0,0](3,1)
\roundnode{m4}(4,1)

\roundnode{m5}(5.5,1)

\roundnode{m6}(7,0)
\rectnode{m7}[0,0](8,0)
\rectnode{m8}[0,0](9,0)
\roundnode{m9}(10,0)
\roundnode{m10}(11,0)
\rectnode{m11}[0,0](12,0)
\rectnode{m12}[0,0](13,0)
\roundnode{m13}(14,0)

\roundnode{m14}(7,2)
\rectnode{m15}[0,0](8.3,2)
\rectnode{m16}[0,0](9.6,2)
\roundnode{m17}(11,2)

\edge{m1}{m2}
\edge{m2}{m3}[\graphlinedash{1 2}]
\edge{m3}{m4}
\edge{m4}{m5}
\edge{m5}{m6}
\edge{m6}{m7}
\edge{m7}{m8}[\graphlinedash{1 2}]
\edge{m8}{m9}
\edge{m9}{m10}
\edge{m10}{m11}
\edge{m11}{m12}[\graphlinedash{1 2}]
\edge{m12}{m13}

\edge{m5}{m14}
\edge{m14}{m15}
\edge{m15}{m16}[\graphlinedash{1 2}]
\edge{m16}{m17}

\autonodetext{m1}[n]{$-d$}
\autonodetext{m4}[n]{$-e$}
\autonodetext{m4}[s]{$w$}
\autonodetext{m10}[n]{$-b$}
\autonodetext{m10}[s]{$v$}
\autonodetext{m13}[n]{$-d'$}

\freetext(8.5,-0.6){$\underset{\text{$k\geq 1$ vertices}}{\underbrace{\hspace{65pt}}}$}
\freetext(8.95,2.6){$\overset{\text{$c_1-1\geq k+1$ vertices}}{\overbrace{\hspace{86pt}}}$}

\freetext(-0.3,1){$L'_1$}
\freetext(15.3,0){$L'_2$}
\freetext(12,2){$L'_3$}

\end{graph}
\end{center}
\caption{\quad The dual tree $\Ga'$ in the case of Theorem~\ref{t:n1>1}}
\label{f:dugr}
\end{figure}
Observe that we always have $e\geq 3$. This is clear if $l(L'_1)>1$,
because $k\geq 1$. On the other hand, if $l(L'_1)=1$ then $n_1=k$, and
since we are assuming $n_1>1$ we have $e=k+1\geq 3$.  Moreover, if
$n_2=1$ then $b=2$ and $l(L'_2)=k+1$, while if $n_2>1$ then $b\geq 3$
and $l(L'_2)\geq k+1$.

{From} now on, we fix an identification of the set of vertices of
$\Ga$ and $\Ga'$ with, respectively, sets of generators for the second
integral homology of $W_{\Gamma}$ and $W_{\Gamma '}$, so that $\Ga$
and $\Ga'$ are the corresponding intersection graphs. Let $R$ be the
smooth 4--manifold of Lemma~\ref{l:ratsur}. We will now define a
characteristic cohomology class $c\in H^2(R;\Z)$ and a contact
structure $\xi$ on $Y$ via a contact surgery diagram as in
Figure~\ref{f:OSzIII2}, and then we will apply Theorem~\ref{t:magic}.
Denote by $h$ and $e_i$ standard generators of $H_2(R; \bfz )$, where
$h$ has square $+1$ and each $e_i$ has square $-1$.  It is easy to see
that under the map induced by the embedding $W_\Ga \cup
W_{\Ga'}\subset R$, up to renaming the $e_i$'s we have:
\begin{itemize}
\item
The central vertex of $\Ga$ goes to $e_1$,
\item
The central vertex of $\Ga'$ goes to $h-e_2-e_3-e_4$,
\item
The first vertex of each leg $L_i$, $i=1,2,3$, goes to a class of the form
$h-e_1-e_{i+1}-\sum_j e_j$.
\item
All the other vertices of $\Ga$ and $\Ga'$ go to classes of the form $e_j-\sum_k e_k$.
\end{itemize}
Denote by $z$ and $z'$, respectively, the central vertices of $\Ga$ and $\Ga'$. 
Let $x_i$, for $i=1,2,3$, be the first vertex of $L_i$,  that is the vertex closest to $z$, 
and by  $y'_i$, for $i=1,2,3$, the last vertex of $L'_i$, that is the vertex of $L'_i$ most distant 
from $z'$. Let $E=\{e_i\}\subset H_2(R;\Z)$ be the set of exceptional classes. Let  
\[
L'_2[k]:=\{\text{first $k$ vertices of $L'_2$}\},
\]
\[
S:=\{e_i\in E\ |\ e_i\cdot u = 0\ \text{for each}\ u\in L'_2[k]\ \text{and}\ 
e_i\cdot u\neq 0\ \text{for some}\ u\in L'_2\setminus L'_2[k]\},
\]
\[
T:=\{e_i\in E\ |\ e_i\cdot u\neq 0\ \text{for some}\ u\in
L_3'\setminus\{y'_3\} \},
\]
and define the class $c\in H^2 (R; \bfz )$ through its Poincar\'e dual by 
\begin{equation}\label{e:magic}
\PD(c) :=
h - \sum_{e\in E} e + 2\sum_{e\in S} e + 2\sum_{e\in T} e.
\end{equation}
Observe that by construction $S\cap T=\emptyset$, which gives $\langle c, e_i\rangle = \pm 1$ 
for every $e_i$. This immediately implies that $c$ is characteristic and $c^2 =\sigma (R)$. 
A tedious but straightforward verification shows that $c$ evaluates on the vertices of 
$\Ga$ and $\Ga'$ as follows:
\begin{itemize}
\item
$\langle c, z\rangle = +1$,
\item
$\langle c, x_1\rangle = x_1\cdot x_1 = -2$,
\item
if $u\in L_1\setminus\{x_1\}$ then $\langle c, u\rangle = u\cdot u + 2$,
\item
$\langle c, x_2\rangle = x_2\cdot x_2 + 2 = -k$,
\item
if $u\in L_2\setminus\{x_2\}$ then $\langle c, u\rangle = -u\cdot u -2$,
\item
$\langle c, x_3\rangle = -x_3\cdot x_3 - 2 = c_1-2$,
\item
$\langle c, z'\rangle = 0$,
\item
if $u\in L_1'$ then $\langle c, u\rangle = u\cdot u+ 2$,
\item
$\langle c, v\rangle = -v\cdot v= b$,
\item
if $u\in L'_2[k]$ then $\langle c, u\rangle = u\cdot u+ 2 = 0$,
\item
if $u\in L'_2\setminus(L'_2[k]\cup\{v\})$ then $\langle c, u\rangle = -u\cdot u- 2$, 
\item
$\langle c, y'_3\rangle = y'_3\cdot y'_3 = -2$,
\item
if $u\in L'_3\setminus\{y'_3\}$ then $\langle c, u\rangle = u\cdot u+ 2 = 0$.
\end{itemize}

Define $\xi$ to be the contact structure given by the contact
surgery diagram of Figure~\ref{f:pres1}. By~\cite{DG1,DG2}, this
diagram is equivalent to a contact surgery diagram as in
Figure~\ref{f:OSzIII2}. Moreover, after converting each contact
framing into a smooth framing, with a little Kirby calculus as
e.g.~in~\cite{OSzII} it is easy to check that the underlying
topological 3--manifold is $Y=\del W_\Ga$.
\begin{figure}[!ht]
\begin{center}
\psfrag{+1}{$+1$}
\psfrag{-1}{$-1$}
\psfrag{2k1}{\small $2k+1$}
\psfrag{dcu}{\small down cusps}
\psfrag{ucu}{\small up cusps}
\psfrag{2c1}{\small $2c_1-1$}
\psfrag{s1}{\tiny $-1$}
\psfrag{b22}{\tiny $2b_2-3$}
\psfrag{bn22}{\tiny $2b_{n_2}-3$}
\psfrag{sdc}{\tiny down cusps}
\psfrag{km1}{\tiny $k-1$}
\psfrag{ak12}{\tiny $2a_{k+1}-3$}
\psfrag{unk}{\tiny unknots}
\psfrag{an12}{\tiny $2a_{n_1}-3$}
\psfrag{suc}{\tiny up cusps}
\epsfig{file=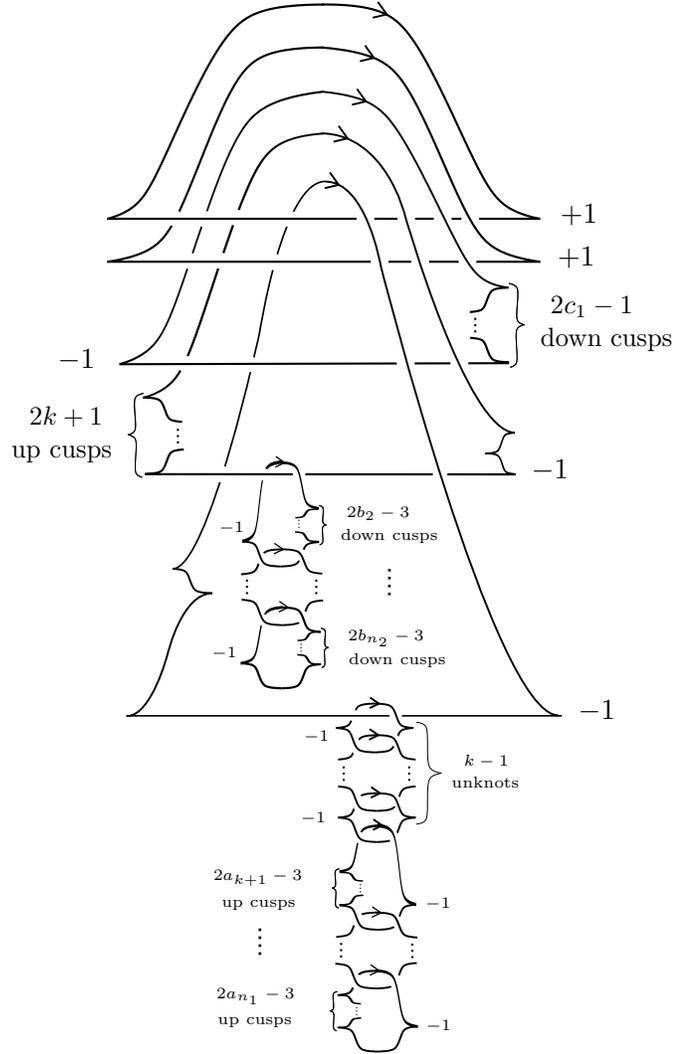, height=14cm}
\end{center}
\caption{The contact structure $\xi$ used in the proof of Theorem~\ref{t:n1>1}.}
\label{f:pres1}
\end{figure}

\begin{lem}\label{l:square1}
We have
\[
d_3(\xi) = \frac{1}{4}\left((c|_{W_{\Ga}})^2-3\si (W_{\Ga})-2 b_2(W_{\Ga})\right)+1.
\]
\end{lem}

\begin{proof}
One can think of Figure~\ref{f:pres1} smoothly, as a handlebody presentation of a smooth 4--manifold $X$ having 
one 0--handle and a number of 2--handles. The knot orientations indicated in Figure~\ref{f:pres1} determine 
rotation numbers which, according to~\cite{Gompf, GS}, can be computed by the formula 
\begin{equation}\label{e:rot}
{\rm {rot}}(K)=\frac 12(c_d -c_u),
\end{equation}
where $c_u $ and $c_d$ denote the number of up and down cusps, respectively, in the
front projection of $K$. Let $\al\in H^2(X;\Z)$ be the unique cohomology class which evaluates on the 
2--homology class corresponding to an oriented knot $K$ of the diagram
as the rotation number of $K$. By~\cite[Corollary~3.6]{DGS} (with the Euler characteristic $\chi$ 
replaced by the second Betti number $b_2$ to adjust for the standard convention in Heegaard Floer theory) 
we have
\[
d_3(\xi) = \frac 14 (\al^2 - 3\si(X) -2 b_2(X)) + 2.
\]
An easy exercise in Kirby calculus shows that there is a orientation--preserving diffeomorphism 
$\varphi\co X\#\CP^2\cong {W_\Ga}\# 2\CP^2$. Therefore we have 
\[
b_2(X) = b_2(W_\Ga) +1\quad\text{and}\quad \si(X) = \si(W_\Ga) + 1. 
\]
Moreover, we can define the extension $\tilde\al\in H^2(X\#\CP^2;\Z)$ of $\al$ by declaring its value on the standard generator of
 the 2--homology of the $\CP^2$--summand to be $-1$. It is easy to check that for a natural choice of $\varphi$ the class  
$\varphi_*(\tilde\al)$ takes value $-1$ on the standard generator of each $\CP^2$--summand, and  
$\varphi_*(\tilde\al) |_{W_\Ga} = c|_{W_{\Ga}}$. Therefore  
\[
(c|_{W_\Ga})^2 + 2 = \varphi_*(\tilde\al)^2 = (\tilde\al)^2 = \al^2 + 1,  
\]
which implies $\al^2 = (c|_{W_\Ga})^2 + 1$. Thus, we conclude 
\begin{multline*}
d_3(\xi) = \frac 14 ((c|_{W_\Ga})^2 + 1 - 3(\si(W_\Ga) + 1) - 2(b_2(W_\Ga) +1)) + 2 = \\
=\frac{1}{4}\left((c|_{W_{\Ga}})^2-3\si (W_{\Ga})-2 b_2(W_{\Ga})\right)+1.
\end{multline*}
\end{proof}

In order to apply Theorem~\ref{t:magic} and conclude that the 
contact structure $\xi$ has non--zero contact invariant, thus proving Theorem~\ref{t:n1>1}, 
it now suffices to check that, when restricted to $W_{\Ga'}$, the class $c$ is contained in a full path on $\Ga'$.  The nonzero values of $c$ on $\Ga'$ are shown in parenthesis in Figure~\ref{f:ert}. 
\begin{figure}[ht]
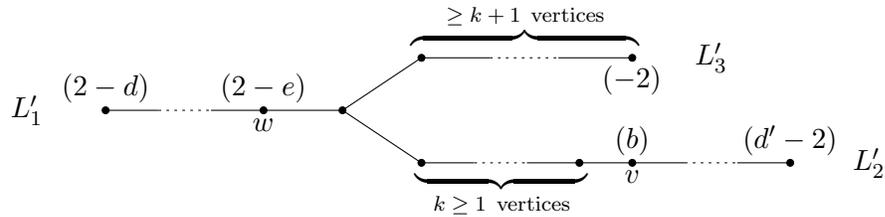

\begin{center}
\setlength{\unitlength}{1mm}
\unitlength=0.7cm
\begin{graph}(14,4)(0,-1)
\graphnodesize{0.15}

\roundnode{m1}(1,1)
\rectnode{m2}[0,0](2,1)
\rectnode{m3}[0,0](3,1)
\roundnode{m4}(4,1)

\roundnode{m5}(5.5,1)

\roundnode{m6}(7,0)
\rectnode{m7}[0,0](8,0)
\rectnode{m8}[0,0](9,0)
\roundnode{m9}(10,0)
\roundnode{m10}(11,0)
\rectnode{m11}[0,0](12,0)
\rectnode{m12}[0,0](13,0)
\roundnode{m13}(14,0)

\roundnode{m14}(7,2)
\rectnode{m15}[0,0](8.3,2)
\rectnode{m16}[0,0](9.6,2)
\roundnode{m17}(11,2)

\edge{m1}{m2}
\edge{m2}{m3}[\graphlinedash{1 2}]
\edge{m3}{m4}
\edge{m4}{m5}
\edge{m5}{m6}
\edge{m6}{m7}
\edge{m7}{m8}[\graphlinedash{1 2}]
\edge{m8}{m9}
\edge{m9}{m10}
\edge{m10}{m11}
\edge{m11}{m12}[\graphlinedash{1 2}]
\edge{m12}{m13}

\edge{m5}{m14}
\edge{m14}{m15}
\edge{m15}{m16}[\graphlinedash{1 2}]
\edge{m16}{m17}

\autonodetext{m1}[n]{$(2-d)$}
\autonodetext{m4}[n]{$(2-e)$}
\autonodetext{m4}[s]{$w$}
\autonodetext{m10}[n]{$(b)$}
\autonodetext{m10}[s]{$v$}
\autonodetext{m13}[n]{$(d'-2)$}
\autonodetext{m17}[s]{$(-2)$}

\freetext(8.5,-0.6){$\underset{\text{$k\geq 1$ vertices}}{\underbrace{\hspace{65pt}}}$}
\freetext(8.95,2.6){$\overset{\text{$\geq k+1$ vertices}}{\overbrace{\hspace{86pt}}}$}

\freetext(-0.5,1){$L'_1$}
\freetext(15.5,0){$L'_2$}
\freetext(12.5,2){$L'_3$}

\end{graph}
\end{center}
\caption{\quad The nonzero values of $c$}
\label{f:ert}
\end{figure}
In what follows it  will be convenient to introduce shorthands to keep track of the values of cohomology classes such as $c$. For example, we can express the information contained in Figure~\ref{f:ert} as follows:
\[
\begin{matrix}
                                 &    &  0 &\cdots & 0 & (-2) &\\
(2-d)\ \cdots\ (2-e)  & 0  &  & & & &\\
                            &  &    0 &\cdots & 0 & b  & \cdots\ (d'-2)
\end{matrix}
\]

\begin{lem}\label{l:path1}
The vector of values defined by $c$ on $\Ga'$ is contained in a full path.
\end{lem}

\begin{proof}
Throughout the proof we will identify, when convenient, characteristic classes with their 
sets of values on the standard homology generators. Observe that the value $-2$ on $y'_3$ 
prevents $c$ from being initial, and the value $b$ on $v$ prevents $c$ from being 
terminal.

We start by showing  that there is a sequence of characteristic classes from $c$ 
to a terminal vector $L$. Recall from Section~\ref{s:HF} that this means 
\[
u\cdot u\leq L\cdot u\leq -u\cdot u - 2
\]
for every vertex $u$. Replacing $c$ with $c+2\PD(v)$ creates value $-b$ on $v$, $+2$ 
on the vertex of position $k$ on $L_2'$ and, if $l(L_2')>k+1$, value $f$ on the vertex to the 
right of $v$, of position $k+2$ on $L_2'$ (assuming its weight is $-f$). 
The resulting set of values can be represented as follows:
\[
\begin{matrix}
                                &  &  0 &\cdots & 0 & (-2) &\\
(2-d)\ \cdots\ (2-e) & 0  &  & & & &\\
                             & &    0 &\cdots & 2 & (-b)  & f\cdots\ (d'-2)
\end{matrix}
\]
By a sequence of similar operations we obtain
\[
\begin{matrix}
                                &  &  0 &\cdots\  0& (-2) & &\\
(2-d)\ \cdots\ (2-e) & 0  &  & & & &\\
                             & &    0 &\cdots\ 2 & (-2) & (2-b)  & f\cdots\ (d'-2)
\end{matrix}
\]
\[
\vdots
\]
\[
\begin{matrix}
                                &  &  0 &\cdots & 0 & (-2) &\\
(2-d)\ \cdots\ (2-e) & 2  &  & & & &\\
                             & &   (-2) &\cdots & 0 & (2-b)  & f\cdots\ (d'-2)
\end{matrix}
\]
Adding twice the Poincar\'e dual of the homology class corresponding to the 
central vertex we get 
\[
\begin{matrix}
                                &  &  2 &\cdots & 0 & (-2) &\\
(2-d)\ \cdots\ (4-e) & (-2)  &  & & & &\\
                             & &  0 &\cdots & 0 & (2-b)  & f\cdots\ (d'-2)
\end{matrix}
\]
Notice that $4-e$ is an admissible value for a class in a full path,
that is  
\[
-e<-e+4<e, 
\]
because we are assuming $e\geq 3$. By another sequence of similar 
operations we arrive at 
\[
\begin{matrix}
                                &  &  0 &\cdots\ (-2) &  2 & (-2) &\\
(2-d)\ \cdots\ (4-e) & 0  &  & & & &\\
                             & &  0 &\cdots & 0 & (2-b)  & f\cdots\ (d'-2)
\end{matrix}
\]
and then at 
\[
\begin{matrix}
                                &  &  0 &\cdots\  0&  (-2) & 0 &\\
(2-d)\ \cdots\ (4-e) & 0  &  & & & &\\
                             & &  0 &\cdots & 0 & (2-b)  & f\cdots\ (d'-2)
\end{matrix}
\]
If $l(L_2')>k+1$ we need to deal with the value $f$ on the vertex at position $k+2$ on $L_2'$.
Recall that $l(L'_2)>k+1$ occurs only if $b\geq 3$. Therefore we can proceed to 
\[
\begin{matrix}
                                &  &  0 &\cdots\  0&  (-2) & 0 &\\
(2-d)\ \cdots\ (4-e) & 0  &  & & & &\\
                             & &  0 &\cdots & 0 & (4-b)  & (-f)\ g\cdots\ (d'-2)
\end{matrix}
\]
where we are assuming that the value of the vertex of position $k+3$ exists and has weight 
$-g$ (if such vertex does not exist, we are done). Notice that the value $4-b$ is good because $b\geq 3$.
If we keep going like this we eventually arrive at the vector $L$ given by 
\[
\begin{matrix}
                                &  &  0 &\cdots\  0&  (-2) & 0 &\\
(2-d)\ \cdots\ (4-e) & 0  &  & & & &\\
                             & &  0 &\cdots & 0 & (4-b)  & (2-f)\ (2-g)\cdots\ (-d')
\end{matrix}
\]
The vector $L$ is terminal in the sense of Section~\ref{s:HF}.  The
above argument shows that, regardless of whether $l(L'_2)=k+1$ or
$l(L'_2)>k+1$, we can always find a path joining $c$ to such a terminal
vector. This concludes the first half of the proof.

In the second part of the proof we show the existence of a path
joining $c$ back to an initial vector $K$ in the sense of
Section~\ref{s:HF}. At each step now we are subtracting, instead of
adding, twice the Poincar\'e dual of a homology class corresponding to
a vertex.  Since there is no new conceptual ingredient involved, we
present just the shorthand description of the various steps,
commenting only when strictly necessary.
\[
c = 
\begin{matrix}
                                 &    &  0 &\cdots & 0 & (-2) &\\
(2-d)\ \cdots\ (2-e)  & 0  &  & & & &\\
                            &  &    0 &\cdots & 0 & b  & \cdots\ (d'-2)
\end{matrix}
\]
\medskip
\[
\lra
\begin{matrix}
                                 &    &  0 &\cdots & (-2) & 2 &\\
(2-d)\ \cdots\ (2-e)  & 0  &  & & & &\\
                            &  &    0 &\cdots & 0 & b  & \cdots\ (d'-2)
\end{matrix}
\]
\medskip
\[
\lra\cdots\lra
\begin{matrix}
                                 &    &  2 &\cdots & 0 & 0 &\\
(2-d)\ \cdots\ (2-e)  & (-2)  &  & & & &\\
                            &  &    0 &\cdots & 0 & b  & \cdots\ (d'-2)
\end{matrix}
\]
\medskip
\[
\lra
\begin{matrix}
                                 &    &  0 &\cdots & 0 & 0 &\\
(2-d)\ \cdots\ (-e)  & 2  &  & & & &\\
                            &  &   (-2) &\cdots & 0 & b  & \cdots\ (d'-2)
\end{matrix}
\]
\medskip
\[
\lra\cdots\lra
\begin{matrix}
                                 &    &  0 &\cdots & 0 & 0 &\\
(2-d)\ \cdots\ (-e)  & 0  &  & & & &\\
                            &  &    0 &\cdots\ 2 & (-2) & b  & \cdots\ (d'-2)
\end{matrix}
\]
\medskip
\[
\lra
\begin{matrix}
                                 &    &  0 &\cdots & 0 & 0 &\\
(2-d)\ \cdots\ (-e)  & 0  &  & & & &\\
                            &  &    0 &\cdots & 2 & (b-2)  & \cdots\ (d'-2)
\end{matrix}
\]
\medskip
\[
\lra
\begin{matrix}
                                 &    &  0 &\cdots & 0 & 0 &\\
(2-d)\ \cdots\ (-h) \ e & (-2)  &  & & & &\\
                            &  &    0 &\cdots & 2 & (b-2)  & \cdots\ (d'-2)
\end{matrix}
\]
where we are assuming that the second vertex of $L_1'$ exists and has weight 
$-h$. Still assuming that $l(L'_1)>1$, we eventually arrive at the 
configuration:
\[
\tag{$A$}
U = 
\begin{matrix}
                                 &    &  0 &\cdots & 0 & 0 &\\
d\ \cdots\ (h-2) \ (e-2) & (-2)  &  & & & &\\
                            &  &    0 &\cdots & 2 & (b-2)  & \cdots\ (d'-2)
\end{matrix}
\]
On the other hand, if $l(L'_1)=1$ at this point we have the set of values:
\[
\tag{$B$}
V = 
\begin{matrix}
                                 &    &  0 &\cdots & 0 & 0 &\\
e & (-2)  &  & & & &\\
                            &  &    0 &\cdots & 2 & (b-2)  & \cdots\ (d'-2)
\end{matrix}
\]
In case $(A)$ there are two subcases: $k=1$ and $k>1$. If $k=1$ we proceed as follows: 
\[
U\  =\ 
\begin{matrix}
                                 &    &  0  & \cdots \ 0 &\\
d\ \cdots\ (h-2) \ (e-2) & (-2)  &  & & \\
                                  &  &    2 &(b-2) & \cdots\ (d'-2) 
\end{matrix}
\]
\medskip
\[
\lra 
\begin{matrix}
                                 &    &  (-2)  & \cdots \ 0 &\\
d\ \cdots\ (h-2) \ (e-4) & 2  &  & & \\
                                    &  &   0 &(b-2) & \cdots\ (d'-2) 
\end{matrix}
\]
\medskip
\[
\lra\cdots\lra
\begin{matrix}
                                 &    &  0  & \cdots \ 2 &\\
d\ \cdots\ (h-2) \ (e-4) & 0  &  & & \\
                                    &  &   0 &(b-2) & \cdots\ (d'-2) 
\end{matrix}
\ =\ K.
\]
Observe that this vector $K$ is initial, therefore in this subcase we are done. Notice that 
in case $(B)$ $e\geq 3$ forces $k>1$, therefore from now on we assume $k>1$ 
in both cases $(A)$ and $(B)$. In case $(A)$ we proceed as follows:
\[
U = 
\begin{matrix}
                                 &    &  0 &\cdots & 0 & 0 &\\
d\ \cdots\ (h-2) \ (e-2) & (-2)  &  & & & &\\
                            &  &    0 &\cdots & 2 & (b-2)  & \cdots\ (d'-2)
\end{matrix}
\]
\medskip
\[
\lra
\begin{matrix}
                                 &    &  (-2) &\cdots & 0 & 0 &\\
d\ \cdots\ (h-2) \ (e-4) & 2  &  & & & &\\
                            &  &    (-2) &\cdots & 2 & (b-2)  & \cdots\ (d'-2)
\end{matrix}
\]
\medskip
\[
\lra\cdots\lra
\begin{matrix}
                                 &    &  (-2) &\cdots & 0 & 0 &\\
d\ \cdots\ (h-2) \ (e-4) & 0  &  & & & &\\
                            &   &  0 &\cdots\ 2 & 0 & (b-2)  & \cdots\ (d'-2)
\end{matrix}
\]
\medskip
\[
\lra
\begin{matrix}
                                 &    &  2 &(-2)\ \cdots & 0 & 0 &\\
d\ \cdots\ (h-2) \ (e-4) & (-2)  &  & & & &\\
                            &   &  0 &\cdots\ 2 & 0 & (b-2)  & \cdots\ (d'-2)
\end{matrix}
\]
\medskip
\[
\lra
\begin{matrix}
                                 &    &  0 &(-2)\ \cdots & 0 & 0 &\\
d\ \cdots\ (h-2) \ (e-6) & 2  &  & & & &\\
                            &   &  (-2) &\cdots\ 2 & 0 & (b-2)  & \cdots\ (d'-2)
\end{matrix}
\]
\medskip
\[
\lra\cdots\lra
\begin{matrix}
                                 &    &  0 &(-2)\ \cdots\  0\ 0 &\\
d\ \cdots\ (h-2) \ (e-6) & 0  &  & & & &\\
                                  &   &  0 &\cdots\ 2\ 0\  0\  (b-2)  & \cdots\ (d'-2)
\end{matrix}
\]
\medskip
\[
\lra\cdots\lra
\begin{matrix}
                                 &    &  \overbrace{0\ \cdots\ (-2)\ 2\ (-2)}^{t+1}\ \cdots\  0\phantom{xxxxx} &\\
d\ \cdots\ (h-2) \ (e-2t-2) & 0  &  & \\
                                  &   & \underbrace{0\ \cdots\ 2}_{k-t}\  \cdots\  (b-2)\ \cdots\ (d'-2) &
\end{matrix}
\]
In case $(B)$ similar steps lead to 
\[
\begin{matrix}
                                 &    &  \overbrace{0\ \cdots\ (-2)\ 2\ (-2)}^{t+1}\ \cdots\ 0 \phantom{xxxxx}&\\
(e-2t) & 0  &  & \\
                                  &   & \underbrace{0\ \cdots\ 2}_{k-t}\  \cdots\  (b-2)\ \cdots\ (d'-2) &
\end{matrix}
\]
Since by assumption $e=k+2$ in case $(A)$ and $e=k+1$ in case $(B)$,
choosing $t=k-1$ at this point leads, in both cases $(A)$ and $(B)$,
to the vector
\[
\begin{matrix}
                                 &    &  \overbrace{0\ \cdots\ (-2)\ 2\ (-2)}^{k}\ \cdots\ 0 \phantom{xxx}&\\
d\ \cdots\ (h-2) \ (4-e) & 0  &  & \\
                                  &   &  2\ 0\ \cdots\  0\  (b-2)\ \cdots\ (d'-2) &
\end{matrix}
\]
Then, we proceed as follow
 
\[
\lra\cdots\lra
\begin{matrix}
                                 &    &  \overbrace{2\ \cdots\ 0\ (-2)}^{k}\ \cdots\ 0 \phantom{xxxxxxx}&\\
d\ \cdots\ (h-2) \ (4-e) & (-2) &  & \\
                                  &   &  2\ 0\ \cdots\  0\  (b-2)\ \cdots\ (d'-2) &
\end{matrix}
\]
\medskip
\[
\lra\cdots\lra
\begin{matrix}
                                 &    &  \overbrace{0\ \cdots\ 0\ (-2)}^{k}\ \cdots\ 0 \phantom{xxxxxx}&\\
d\ \cdots\ (h-2) \ (2-e) & 2 &  & \\
                                  &   &  0\ \cdots\  0\  (b-2)\ \cdots\ (d'-2) &
\end{matrix}
\]
\medskip
\[
\lra
\begin{matrix}
                                 &    &  \overbrace{0\ \cdots\ (-2)\ (2)}^{k}\ (-2)\cdots\ 0 \phantom{xx}&\\
d\ \cdots\ (h-2) \ (2-e) & 2 &  & \\
                                  &   &  0\ \cdots\  0\  (b-2)\ \cdots\ (d'-2) &
\end{matrix}
\]
\medskip
\[
\lra\cdots\lra
\begin{matrix}
                                 &    &  2 \ \overbrace{0\ \cdots\ 0\ (-2)}^k \ 0\cdots\ 0 \phantom{xxx}& \\
d\ \cdots\ (h-2) \ (2-e) & 0 &  & \\
                                  &   &  0\ \cdots\  0\ (b-2) \ \cdots \  (d'-2) &
\end{matrix}
\]
\medskip
\[
\lra\overset{(\text{$t-1$ steps})}{\cdots}\lra
\begin{matrix}
                                 &    & \overbrace{0\ \cdots\ 0\ 2}^{t}\  \overbrace{0\ \cdots\ 0\ (-2)}^{k}\ 0\cdots\ 0 &\\
d\ \cdots\ (h-2) \ (2-e) & 0 &  & \\
                                  &   &  0\ \cdots\  0\  (b-2)\ \cdots\ (d'-2) \phantom{xxx}&
\end{matrix}
\]
\medskip
\[
\lra\cdots\lra
\begin{matrix}
                                 &    & \overbrace{0\ \cdots\ 0\ 2}^{l(L'_3)-k}\ \overbrace{0\ \cdots\ 0\ (-2)}^{k} \phantom{xxx}&\\
d\ \cdots\ (h-2) \ (2-e) & 0 &  & \\
                                  &   &  0\ \cdots\  0\  (b-2)\ \cdots\ (d'-2) &
\end{matrix}
\]
\medskip
\[
\lra
\begin{matrix}
                                 &    & \overbrace{0\ \cdots\ 0\ 2}^{l(L'_3)-k}\ \overbrace{0\ \cdots\ (-2)\ 2}^{k} \phantom{xx}&\\
d\ \cdots\ (h-2) \ (2-e) & 0 &  & \\
                                  &   &  0\ \cdots\  0\  (b-2)\ \cdots\ (d'-2) &
\end{matrix}
\]
\medskip
\[
\lra\cdots\lra
\begin{matrix}
                                 &    & 0 & \cdots & 0 & \overbrace{2\ \cdots\ 0}^{k} & &\\
d\ \cdots\ (h-2) \ (2-e) & 0 &  & & & & &\\
                                  &   &  0 & \cdots & 0 & (b-2)\ \cdots & (d'-2) & 
\end{matrix}
\ =\ K
\]
Clearly $K$ is an initial vector, so this concludes the proof of the existence of the full path.
\end{proof}

\begin{proof}[Proof of Theorem~\ref{t:n1>1}]
By the assumptions on $Y$ and Lemmas~\ref{l:square1} and~\ref{l:path1}, Theorem~\ref{t:magic} 
applies. Therefore we conclude that the contact structure $\xi$ defined in Figure~\ref{f:pres1} 
has nonzero contact invariant.
\end{proof}

\section{Second application of the criterion}
\label{s:secondappl}

In this section we apply Theorem~\ref{t:magic} to prove Theorem~\ref{t:n1=1} below. 

\begin{thm}\label{t:n1=1}
Let $Y = Y(-1;r_1,r_2,r_3)$ with $e(Y) > 0$, and suppose that $Y$ carries no transverse contact 
structures. Suppose that $1>r_1\geq r_2\geq r_3>0$ satisfy Expansions~\eqref{e:cfracs} and 
each of the following holds:
\begin{itemize}
\item
$n_1=1$ and $a_1=2$;
\item
$c_1\geq b_1=3$; 
\item
$n_3=1$.
\end{itemize}
Then, either $Y\cong M_n$ for some $n\geq 1$ or  $Y$ carries a contact structure 
$\xi$ given by a surgery diagram as in Figure~\ref{f:OSzIII2} and such that $c(Y,\xi)\neq 0$. 
\end{thm}

Under the assumptions of Theorem~\ref{t:n1=1} we have $Y=W_\Ga$, where the tree $\Ga$ takes the form given in Figure~\ref{f:subcase0}.
\begin{figure}[ht]
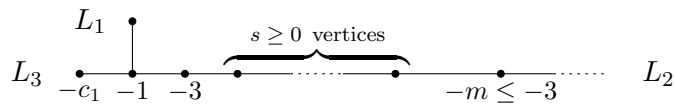

\begin{center}
\setlength{\unitlength}{1mm}
\unitlength=0.7cm
\begin{graph}(11,1)(-2,0)
\graphnodesize{0.15}

\roundnode{n1}(-1,0)
\roundnode{n2}(0,0)
\roundnode{n3}(1,0)
\roundnode{n4}(2,0)
\rectnode{n5}[0,0](3,0)
\rectnode{n6}[0,0](4,0)
\roundnode{n7}(5,0)
\roundnode{n8}(7,0)
\roundnode{n9}(0,1)
\rectnode{n10}[0,0](8,0)
\rectnode{n11}[0,0](9,0)

\edge{n1}{n2}
\edge{n2}{n3}
\edge{n3}{n4}
\edge{n4}{n5}
\edge{n5}{n6}[\graphlinedash{1 2}]
\edge{n6}{n7}
\edge{n7}{n8}
\edge{n8}{n10}
\edge{n10}{n11}[\graphlinedash{1 2}]
\edge{n2}{n9}

\autonodetext{n1}[s]{\small $-c_1$} 
\autonodetext{n2}[s]{\small $-1$} 
\autonodetext{n3}[s]{\small $-3$}
\freetext(3.5,0.53){$\overset{\text{$s\geq 0$ vertices}}{\overbrace{\hspace{70pt}}}$} 
\autonodetext{n8}[s]{\small $-m\leq -3$}

\freetext(-2,0){$L_3$}
\freetext(-0.8, 1){$L_1$}
\freetext(10,0){$L_2$}

\end{graph}
\end{center}
\caption{\quad The tree $\Ga$ under the assumptions of Theorem~\ref{t:n1=1}}
\label{f:subcase0}
\end{figure}

\begin{lem}\label{l:last}
To prove Theorem~\ref{t:n1=1} it suffices to prove the statement under the 
following extra assumptions: 
\begin{itemize}
\item
$c_1=2s+5$;
\item
$n_2=s+2$,
\end{itemize}
where $s\geq 0$ is the number of $(-2)$--vertices on $L_2$ after the first vertex. 
\end{lem}

\begin{proof}
Let $Y=Y(-1;r_1,r_2,r_3)$ with $e(Y)>0$, suppose that $Y$ carries no transverse contact 
structures and the $r_i$ satisfy the assumptions of Theorem~\ref{t:n1=1}. 
We will argue as in Proposition~\ref{p:reduction2}, viewing the weighted tree as prescribing an 
integral surgery presentation for $Y=\del W_\Ga$. We successively
blow down $(-1)$--framed unknots by starting with the central one, continuing  
with the unique vertex of $L_1$ and then with the first $s+1$ vertices on $L_2$. 
In this way we obtain a link consisting of a $(2,
2s+3)$--torus knot, linked positively twice to an unknot if $n_2>s+1$, 
plus a further chain of unknots if $n_2>s+2$.

By~\cite{OSzI} we know that if the smooth surgery coefficient $\ga =
-c_1+2+4(s+1)$ on the $(2,2s+3)$--torus knot $T_{2,2s+3}$ is not equal
to $2s+1$, then the 3--manifold $N=S^3_\ga(T_{2,2s+3})$ carries a
contact structure $\ze$ with $c(N,\ze)\neq 0$ obtained by contact
$(\ga-2s-1)$--surgery on a Legendrian $(2,2s+3)$--torus knot with
Thurston--Bennequin invariant $2s+1$ in the standard contact
$S^3$. Since the smooth framings of the unknots are all $\leq -2$ and
all linking numbers are non--negative, the corresponding topological
surgeries can all be realized by contact $(-1)$--surgeries. Thus,
arguing as in Proposition~\ref{p:reduction2} we see that if $\ga\neq
2s+1$ (equivalently, $c_1\neq 2s+5$) then $Y$ carries a contact
structure $\xi$ such that $c(Y,\xi)\neq 0$. Since by assumption $Y$
carries no transverse contact structures, by~\cite[Theorem~1.2]{Ghigg}
the contact structure $\xi$ has maximal twisting $t(\xi)=0$. This
implies, by~\cite[Proposition~6.1]{OSzIII}, that $\xi$ is given by a
surgery diagram as in Figure~\ref{f:OSzIII2}.  Therefore, the
conclusion of Theorem~\ref{t:n1=1} holds for $Y$ and we see that the
extra assumption $c_1=2s+5$ leads to no loss of generality. Moreover,
if $c_1=2s+5$ and $n_2=s+1$ then $Y\cong M_{n_2}$.  Thus, we may
assume without loss that $c_1=2s+5$ and $n_2> s+1$.  

Finally, we observe that it suffices to assume $n_2=s+2$. In fact,
suppose that each manifold $Z$ corresponding to $n_2=s+2$ carries a
contact structure $\eta$ given by a surgery diagram as in
Figure~\ref{f:OSzIII2} with $c(Z,\eta)\neq 0$. Then, arguing as before
it is easy to see that each manifold $Y$ corresponding to $n_2>s+2$
can be obtained as the underlying 3--manifold of a contact
$(-1)$--surgery on some $(Z,\eta)$. This concludes the proof.
\end{proof}

By Lemma~\ref{l:last}, in order to finish the proof of Theorem~\ref{t:n1=1} 
it suffices to consider trees $\Ga$ as in Figure~\ref{f:subcase1} and 
dual trees $\Ga'$ as in Figure~\ref{f:subcase2}, where
unmarked vertices have weight $-2$.

\begin{figure}[ht]
\begin{center}
\setlength{\unitlength}{1mm}
\unitlength=0.7cm
\begin{graph}(9,1)(-2,0)
\graphnodesize{0.15}

\roundnode{n1}(-1,0)
\roundnode{n2}(1,0)
\roundnode{n3}(2,0)
\roundnode{n4}(3,0)
\rectnode{n5}[0,0](4,0)
\rectnode{n6}[0,0](5,0)
\roundnode{n7}(6,0)
\roundnode{n8}(7,0)
\roundnode{n9}(1,1)

\edge{n1}{n2}
\edge{n2}{n3}
\edge{n3}{n4}
\edge{n4}{n5}
\edge{n5}{n6}[\graphlinedash{1 2}]
\edge{n6}{n7}
\edge{n7}{n8}
\edge{n2}{n9}

\autonodetext{n1}[sw]{\small $-2s-5$} \autonodetext{n2}[s]{\small
$-1$} \autonodetext{n3}[s]{\small $-3$}
\freetext(4.5,0.53){$\overset{\text{$s\geq 0$
vertices}}{\overbrace{\hspace{70pt}}}$} \autonodetext{n8}[se]{\small
$-m\leq -3$}

\freetext(-1.5,0.5){$L_3$}
\freetext(1.8,1){$L_1$}
\freetext(7.5,0.5){$L_2$}

\end{graph}
\end{center}
\caption{The tree $\Ga$ under the assumptions of Theorem~\ref{t:n1=1} and Lemma~\ref{l:last}}
\label{f:subcase1}
\end{figure}

\begin{figure}[ht]
\begin{center}
\setlength{\unitlength}{1mm}
\unitlength=0.7cm
\begin{graph}(12,1)(-1,0)
\graphnodesize{0.15}

\roundnode{m1}(0,0)
\rectnode{m2}[0,0](1,0)
\rectnode{m3}[0,0](2,0)
\roundnode{m4}(3,0)
\roundnode{m5}(4,0)  
\roundnode{m6}(5,0)
\roundnode{m7}(6.5,0)
\roundnode{m8}(8,0)
\rectnode{m9}[0,0](9,0)
\rectnode{m10}[0,0](10,0)
\roundnode{m11}(11,0)
\roundnode{m12}(4,1)

\edge{m1}{m2}
\edge{m2}{m3}[\graphlinedash{1 2}]
\edge{m3}{m4}
\edge{m4}{m5}
\edge{m5}{m6}
\edge{m6}{m7}
\edge{m7}{m8}
\edge{m8}{m9}
\edge{m9}{m10}[\graphlinedash{1 2}]
\edge{m10}{m11}
\edge{m5}{m12}

\autonodetext{m7}[s]{\small $-s-3$}
\autonodetext{m7}[n]{\small $v$}
\freetext(1.5,0.55){$\overset{\text{$2(s+2)$
vertices}}{\overbrace{\hspace{65pt}}}$}
\freetext(9.5,0.6){$\overset{\text{$m-2$
vertices}}{\overbrace{\hspace{65pt}}}$}

\freetext(-1,0){$L'_3$}
\freetext(4.8,1){$L'_1$}
\freetext(12,0){$L'_2$}

\end{graph}
\end{center}
\caption{The tree $\Ga'$ under the assumptions of Theorem~\ref{t:n1=1} and Lemma~\ref{l:last}}
\label{f:subcase2}
\end{figure}

By Lemma~\ref{l:ratsur} there exists a smooth, closed 4--manifold $R$ containing $Y$ as a hypersurface, with an open tubular neighborhood $\nu(Y)\subset R$ such that  $R\setminus \nu(Y)=W_\Ga\cup W_{\Ga'}$. 
Arguing as in the previous section, we now define a characteristic cohomology class $c\in H^2(R;\Z)$ and a contact structure $\xi$ on $Y$ given by a contact surgery diagram as in Figure~\ref{f:OSzIII2}. Then, we will apply Theorem~\ref{t:magic}.

Fix an identification of the set of vertices of $\Ga$ and $\Ga'$ with, 
respetively, sets of generators for the second integral homology of $W_{\Gamma}$ and $W_{\Gamma '}$,
so that $\Ga$ and $\Ga'$ are the corresponding intersection graphs. Denote by $h$ and $e_i$ standard generators of $H_2(R; \bfz )$, where $h$ has square $+1$ and each $e_i$ has square $-1$. Under the map induced by the embedding $W_\Ga' \cup W_{\Ga'}\subset R$, up to renaming the $e_i$'s we have:
\begin{itemize}
\item
The central vertex of $\Ga$ goes to $e_1$,
\item
The central vertex of $\Ga'$ goes to $h-e_2-e_3-e_4$,
\item
The first vertex $x_i$ of each leg $L_i$, $i=1,2,3$, goes to a class of the form
$h-e_1-e_{i+1}-\sum_j e_j$.
\item
All the other vertices go to classes of the form $e_j-\sum_i e_i$.
\end{itemize}
Denote by $x'_2$ the first vertex of $L'_2$, by $y'_3$ the last vertex
of $L'_3$ and by $v$ the vertex indicated in Figure~\ref{f:subcase2}. 
Let $E=\{e_i\}$ be the set of exceptional classes. We define a 
class $c\in H^2(R;\Z)$ by the formula 
\[
\PD(c) := h-\sum_{e\in E} e + 2\sum_{e\in T} e + 2\sum_{e\in S'} e,
\]
where $E$ is the set of exceptional classes, 
\[
T:=\{e_i\in E\ |\ e_i\cdot u\neq 0\ \text{for some}\ u\in
L_3'\setminus\{y'_3\} \}
\]
and
\[
S' := \{e_i\in E\ |\ e_i\cdot u\neq 0\ \text{for every}\ u\in L'_2\setminus\{v, x'_2\}\ \text{and}\ e_i\cdot x'_2=e_i\cdot v=0\}.
\]
Observe that by construction $S\cap T=\emptyset$, which gives $\langle c, e_i\rangle = \pm 1$ 
for every $e_i$. This immediately implies that $c$ is characteristic and $c^2 =\sigma (R)$. 
When restricted to $W_{\Ga}$, the class $c$ takes the values given by  
\[
\begin{matrix}
                  &  (-2)   &                  \\
(2s+3)   &  1    & (-3)\ 0\ \cdots\ 0\ (m-2).
\end{matrix}
\]

Define $\xi$ to be the contact structure defined by the contact surgery diagram of 
Figure~\ref{f:pres2}. By~\cite{DG1,DG2}, this diagram is equivalent to a contact surgery diagram 
as in Figure~\ref{f:OSzIII2}. Moreover, with a little Kirby calculus as in Section~\ref{s:firstappl} 
it is easy to check that the underlying topological 3--manifold is $Y=\del W_\Ga$. 
\begin{figure}[!ht]
\begin{center}
\psfrag{+1}{$+1$}
\psfrag{-1}{$-1$}
\psfrag{k}{\small $k$}
\psfrag{sta}{\small down cusps}
\psfrag{2s4}{\small $4s+9$}
\psfrag{s}{\small $s$}
\psfrag{s1}{\tiny $-1$}
\psfrag{sst}{\small down cusps}
\psfrag{unk}{\small unknots}
\psfrag{m-2}{\small $2m-3$}
\epsfig{file=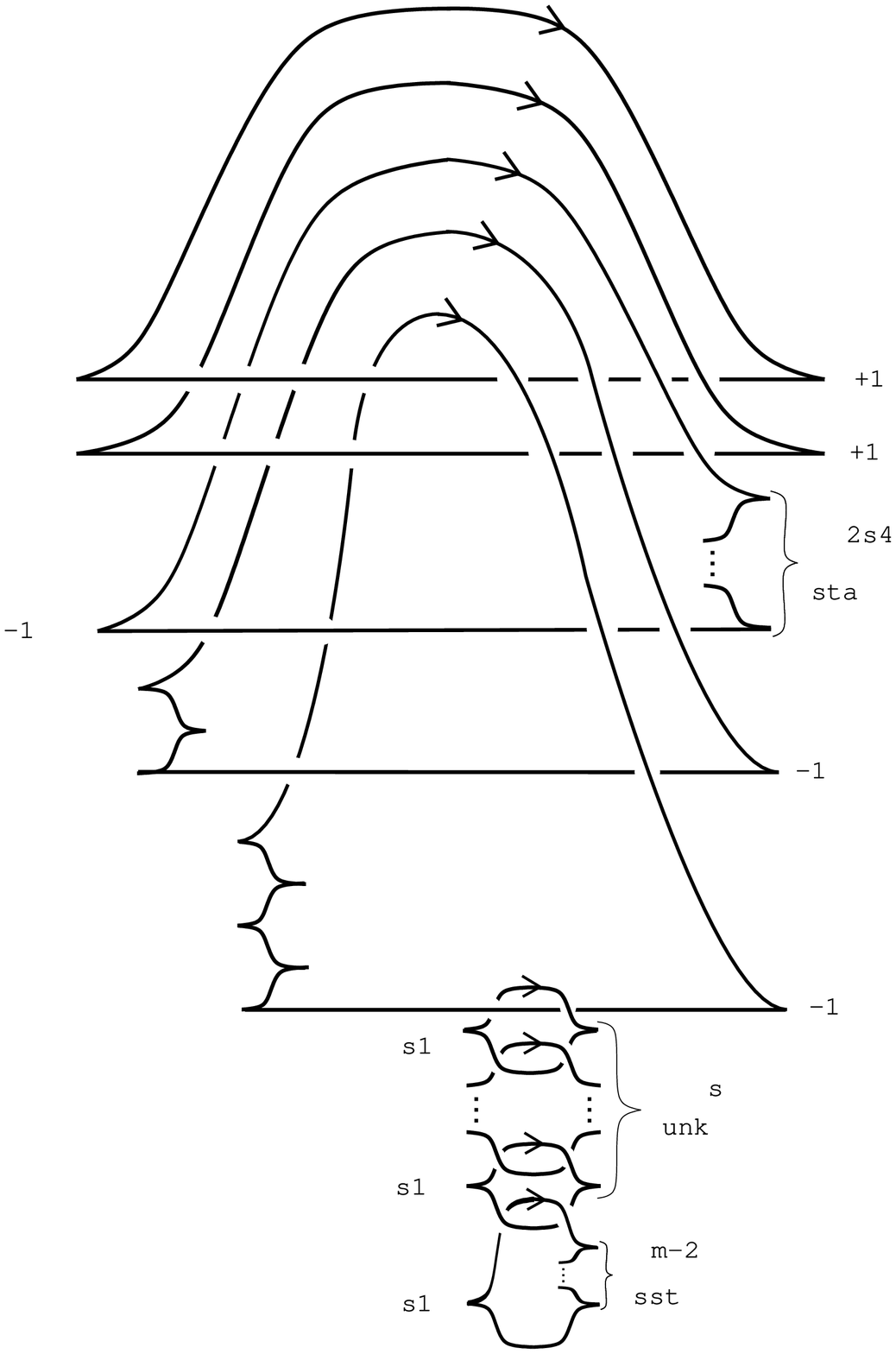, height=11cm}
\end{center}
\caption{The contact structure $\xi$ used in the proof of Lemma~\ref{l:last}}
\label{f:pres2}
\end{figure}

\begin{lem}\label{l:square2}
We have 
\[
d_3(\xi) = \frac{1}{4}\left((c|_{W_{\Ga}})^2-3\si (W_{\Ga})-2 b_2 (W_{\Ga})\right)+1.
\]
\end{lem}

\begin{proof}
The proof is exactly the same as the proof of Lemma~\ref{l:square1},
so here we only outline the argument. When viewed smoothly,
Figure~\ref{f:pres2} gives a handlebody presentation of a smooth
4--manifold $X$ with one 0--handle and a number of 2--handles.  The
rotation numbers associated with every Legendrian knot determine a
cohomology class $\al\in H^2(Z;\Z)$ such that
\begin{multline}
d_3(\xi) = \frac 14 (\al^2 - 3\si(X) -2 b_2(X)) + 2 = \\
= \frac{1}{4}\left((c|_{W_{\Ga}})^2-3\si (W_{\Ga})-2 b_2 (W_{\Ga})\right)+1.
\end{multline}
\end{proof}

Now we want to check that, when restricted to $W_{\Ga'}$, the class $c$ gives a vector of 
values $V$ contained in a full path on $\Ga'$. We have:
\[
V\ =\ 
\begin{matrix}
                               &  0   &                  \\
(-2) \  0\ \cdots\  0    &  0    & 0\ (-s -1) \  2 \ 0 \cdots\ 0
\end{matrix}
\]

\begin{lem}\label{l:path2}
The vector $V$ of values defined by $c$ on $\Ga'$ is contained in a full path.
\end{lem}

\begin{proof}
Observe that the values preventing $V$ from being a terminal, respectively an intial vector are 
$2$ and, respectively $-2$. We start by constructing the lower part of the full path from $V$ to a terminal vector $L$: 
\[
V = 
\begin{matrix}
                               &  0   &                  \\
(-2) \  0\ \cdots\  0    &  0    & 0\ (-s -1) \  2 \ 0 \cdots\ 0
\end{matrix}
\]
\medskip
\[
\lra
\begin{matrix}
                              &  0  &   \\
(-2) \  0\ \cdots\  0    &  0  &  0\ (-s+1) \  (-2) \ 2\ \cdots\ 0
\end{matrix}
\]
\medskip
\[
\lra\cdots\lra
\begin{matrix}
                              &  0  &   \\
(-2) \  0\ \cdots\  0    &  0  &  0\ (-s+1)\ 0\ \cdots\ (-2)\ 2
\end{matrix}
\]
\medskip
\[
\lra
\begin{matrix}
                              &  0  &   \\
(-2) \  0\ \cdots\  0    &  0  &  0\ (-s+1) \ 0\ \cdots\ 0\ (-2)
\end{matrix}
\ =\ L
\]
Observe that $L$ is an terminal vector because, since $s\geq 0$, we have 
\[
-s-3<-s+1<s+3. 
\]
Now we construct the upper part of the full path, connecting $V$ to an initial vector $K$: 
\[
V = 
\begin{matrix}
                              &  0  &   \\
(-2) \  0\ \cdots\  0    &  0  &  0\ (-s-1) \  2\  \cdots\ 0
\end{matrix}
\]
\medskip
\[
\lra
\begin{matrix}
                              &  0  &   \\
2 \ (-2)\ \cdots\  0    &  0  &  0\ (-s-1) \  2\  \cdots\  0
\end{matrix}
\]
\medskip
\[
\lra\cdots\lra
\begin{matrix}
                              &  0  &   \\
0 \ 0 \ \cdots\   2    &  (-2)  &  0\ (-s-1) \  2\  \cdots\  0
\end{matrix}
\]
\medskip
\[
\lra
\begin{matrix}
                              &  (-2)  &   \\
0 \ 0 \ \cdots\   0    &   2  &  (-2)\ (-s-1) \  2\  \cdots\ 0
\end{matrix}
\]
\medskip
\[
\lra
\begin{matrix}
                              &  (-2)  &   \\
0 \ 0 \ \cdots\   0    &   0  &  2\ (-s-3) \  2\  \cdots\ 0
\end{matrix}
\]
\medskip
\[
\lra
\begin{matrix}
                              &  (-2)  &   \\
0 \ 0 \ \cdots\   0    &   0  &  0\ (s+3) \  0\  \cdots\ 0
\end{matrix}
\]
\medskip
\[
\lra
\begin{matrix}
                              &  2  &   \\
0 \ 0 \ \cdots\   0    &   (-2)  &  0\ (s+3) \  0\  \cdots\ 0
\end{matrix}
\]
\medskip
\[
\lra
\begin{matrix}
                              &  0  &   \\
0 \ 0 \ \cdots\   (-2)  &   2  &  (-2)\ (s+3) \  0\  \cdots\ 0
\end{matrix}
\]
\medskip
\[
\lra\cdots\lra
\begin{matrix}
                              &  0  &   \\
2 \ 0 \ \cdots\   0  &   0  &  (-2)\ (s+3) \  0\  \cdots\ 0
\end{matrix}
\]
\medskip
\[
\lra
\begin{matrix}
                              &  0  &   \\
2 \ 0 \ \cdots\   0  &  (-2) &  2\ (s+1) \  0\  \cdots\ 0
\end{matrix}
\]
\medskip
\[
\lra
\begin{matrix}
                            & (-2)  &   \\
2 \ 0 \ \cdots\   (-2) &   2 &  0\ (s+1) \  0\  \cdots\ 0
\end{matrix}
\]
\medskip
\[
\lra\cdots\lra
\begin{matrix}
                            & (-2)  &   \\
0 \ 2 \ \cdots\   0 &   0 &  0\ (s+1) \  0\  \cdots\ 0
\end{matrix}
\]
\medskip
\[
\lra
\begin{matrix}
                            & 2  &   \\
0 \ 2 \ \cdots\   0 &  (-2) &  0\ s+1 \  0\  \cdots\ 0
\end{matrix}
\]
\medskip
\[
\lra
\begin{matrix}
                            & 0  &   \\
0 \ 2 \ \cdots\   (-2) &  2 &  (-2)\ s+1 \  0\  \cdots\ 0
\end{matrix}
\]
\medskip
\[
\lra\cdots\lra
\begin{matrix}
                             &  0 &   \\
0 \ 0 \ 2\cdots\   0 &  0 &  (-2)\ (s+1) \  0\  \cdots\ 0
\end{matrix}
\]
\medskip
\[
\lra
\begin{matrix}
                             &  0 &   \\
0 \ 0 \ 2\cdots\   0 &  (-2) &  2\ (s-1) \  0\  \cdots\ 0
\end{matrix}
\]
\medskip
\[
\lra
\begin{matrix}
                             &  (-2) &   \\
0 \ 0 \ 2\cdots\   (-2) &  2 &  0\ (s-1) \  0\  \cdots\ 0
\end{matrix}
\]
\medskip
\[
\lra\cdots\lra
\begin{matrix}
                                                             &  (-2) &   \\
\underbrace{0\ \cdots\ 2}_{2t+1}\cdots\   (-2) &  2 &  0\ (s+1-2t) \  0\  \cdots\ 0
\end{matrix}
\]
For $t=s+1$ we get
\[
\begin{matrix}
                       &  (-2) &   \\
0\ \cdots\ 2\ (-2) &  2 &  0\ (-s-1) \  0\  \cdots\ 0
\end{matrix}
\]
\medskip
\[
\lra
\begin{matrix}
                       &  (-2) &   \\
0\ \cdots\ 0\ 2 &  0 &  0\ (-s-1) \  0\  \cdots\ 0
\end{matrix}
\]
\medskip
\[
\lra
\begin{matrix}
                       &  2 &   \\
0\ \cdots\ 0\ 2 &  (-2) &  0\ (-s-1) \  0\  \cdots\ 0
\end{matrix}
\]
\medskip
\[
\lra
\begin{matrix}
                       &  0 &   \\
0\ \cdots\ 0\ 0 &  2 &  (-2)\ (-s-1) \  0\  \cdots\ 0
\end{matrix}
\]
\medskip
\[
\lra
\begin{matrix}
                       &  0 &   \\
0\ \cdots\ 0\ 0 &  0 &  2\ (-s-3) \  0\  \cdots\ 0
\end{matrix}
\]
\medskip
\[
\lra
\begin{matrix}
                       &  0 &   \\
0\ \cdots\ 0\ 0 &  0 &  0\ (s+3) \  (-2)\  \cdots\ 0
\end{matrix}
\]
\medskip
\[
\lra
\begin{matrix}
                       &  0 &   \\
0\ \cdots\ 0\ 0 &  0 &  0\ (s+1) \  (2)\ (-2)\ \cdots\ 0
\end{matrix}
\]
\medskip
\[
\lra\cdots\lra
\begin{matrix}
                       &  0 &   \\
0\ \cdots\ 0\ 0 &  0 &  0\ (s+1) \  0\ \cdots\ 0\ 2
\end{matrix}
=K.
\]
\end{proof}

\begin{proof}[Proof of Theorem~\ref{t:n1=1}]
By the assumptions on $Y$ and Lemmas~\ref{l:square2} and~\ref{l:path2}, Theorem~\ref{t:magic} 
applies. Therefore we conclude that the contact structure $\xi$ given in Figure~\ref{f:pres2} 
has non--zero contact invariant.
\end{proof}

\section{The proof of Theorem~\ref{t:main}}
\label{s:final}

In this section we use Proposition~\ref{p:reduction2} and Theorems~\ref{t:n1>1} and~\ref{t:n1=1} 
to prove Theorem~\ref{t:main}. 

Before we start with the proof we need an auxiliary result. 
Let $\Ga$ be the weighted tree of Figure~\ref{f:specialtree}, and let $\tilde\Ga$ denote the 
tree obtained by erasing all vertices of $\Ga$ on the third leg $L_3$ except the first vertex. 
In other words, $\tilde\Ga$ is obtained by  truncating $L_3$ so the resulting leg has length $1$. 
Let $Y=\del W_\Ga$ and $\tilde Y = \del W_{\tilde\Ga}$ be the resulting 3--manifolds. 

\begin{lem}\label{l:legchopped}
Suppose that $Y=\partial W_\Ga$ carries no transverse contact
structures. Then, $\tilde Y =\partial W_{\tilde\Ga}$ carries no
transverse contact structures and $e(\tilde Y)>0$.
\end{lem}

\begin{proof}
We have $Y=Y(-1;r_1,r_2,r_3)$ for some $r_i\in (0,1)\cap\Q$, $i=1,2,3$. 
Recall that by~\cite[Theorem~1.3(c)]{LiMa} the nonexistence of transverse contact structures on 
$Y$ implies that the triple $(r_1,r_2,r_3)$ is not~\emph{realizable}, i.e.~there are no coprime 
integers $m>a>0$ such that, assuming $r_1\geq r_2 \geq r_3$,
\[
\frac 1{r_1} > \frac{m}{a},\quad 
\frac 1{r_2} > \frac{m}{m-a},\quad\text{and}\quad  
\frac 1{r_3} > m.
\]
We also have $\tilde Y=Y(-1;r'_1,r'_2,r'_3)$, where the vector $(r_1', r_2', r_3')$ 
is easily determined to be: 
\[
r_1'= r_1,\quad 
r_2'= r_2,\quad\text{and}\quad
r_3' = \frac 1{c_1},
\]
where $c_1$ is the first continued fraction coefficient of $\frac{1}{r_3}$. We claim that 
the triple $(r_1', r_2', r_3')$ is not realizable. In fact, a pair of coprime integers $m'>a'>0$ with
\[
\frac 1{r'_1} > \frac{m'}{a'},\quad 
\frac 1{r'_2} > \frac{m'}{m'-a'},\quad\text{and}\quad  
\frac 1{r'_3} > m',
\]
would show that $(r_1,r_2,r_3)$ is also realizable, because 
$c_1 = \frac 1{r'_3} > m'$ implies 
\[
\frac{1}{r_3} = [c_1,\ldots, c_{n_3}] > c_1-1 \geq m'.
\]
Since $(r_1', r_2', r_3')$ is not realizable, by~\cite{LiMa} $\tilde Y$ admits 
no transverse contact structures and $e(\tilde Y)\geq 0$. Moreover,  if $e(\tilde Y)=0$ 
then by~\cite{EHN} $\tilde Y$ admits smooth transverse foliations. The fact that $\tilde Y$ 
has a fibration with $3$ multiple fibers implies that $\tilde Y\not\cong S^2\x S^1$, therefore
the results of~\cite{ET} can be applied. Hence, if $e(\tilde Y)=0$, any smooth transverse foliation on $\tilde Y$ could be approximated by a transverse contact structure. 
We conclude that $e(\tilde Y) > 0$ and the lemma is proved.
\end{proof}

\begin{proof}[Proof of Theorem~\ref{t:main}]
By Proposition~\ref{p:reduction2}, to prove Theorem~\ref{t:main} we may assume that $Y=\del W_\Ga$, 
where $\Ga$ is the weighted tree of Figure~\ref{f:specialtree}. More precisely, we may assume: 
\begin{enumerate}
\item
$Y\cong Y(-1;r_1,r_2,r_3)$ and $e(Y)>0$;
\item
$a_1=\cdots=a_k=2$ for some $k\geq 1$ and either 
\begin{itemize}
\item
$n_1=k$ or 
\item
$n_1>k$ and $a_{k+1}>2$;
\end{itemize}
\item
$c_1\geq b_1=k+2$. 
\end{enumerate}
On the other hand, since transverse contact structures are symplectically 
fillable~\cite{LiMa} and therefore tight, to prove Theorem~\ref{t:main} we may also assume: 
\begin{itemize}
\item[(4)]
$Y$ does not carry transverse contact structures. 
\end{itemize}
Now consider the 3--manifold $\tilde Y$ defined above. We claim that if there exists a contact 
structure $\tilde\xi$ on $\tilde Y$ given by a contact surgery as in Figure~\ref{f:OSzIII2} with contact 
invariant $c(\tilde Y,\tilde\xi)\neq 0$, then there exists a contact structure $\xi$ on the 3--manifold $Y
$ satisfying (1)--(4) above, with $c(Y,\xi)\neq 0$. In fact, since all the vertices erased from $\Ga$ to 
obtain $\tilde\Ga$ have weights $\leq -2$ then, under the above assumptions on $(\tilde Y,\tilde\xi)$, as 
in the proof of Proposition~\ref{p:reduction2} we can perform suitable contact $(-1)$--surgeries on 
$(\tilde Y,\tilde\xi)$ to obtain $(Y,\xi)$ with $c(Y,\xi)\neq 0$. 

We observe that, by construction, the 3--manifold $\tilde Y$ satisfies either the assumptions of 
Theorem~\ref{t:n1>1} or those of Theorem~\ref{t:n1=1}, depending on whether the first leg $\tilde L_1$ 
of the tree $\tilde\Ga$ has length, respectively, bigger than $1$ or equal to $1$. If the length of 
$\tilde L_1$ is bigger than $1$ then Theorem~\ref{t:n1>1} applies and we are done. If the length of 
$\tilde L_1$ is equal to $1$ then by Theorem~\ref{t:n1=1} either  $\tilde Y\cong M_n$ for some 
$n\geq 1$ or $\tilde Y$ carries a contact structure $\tilde\xi$ given by a surgery diagram as in 
Figure~\ref{f:OSzIII2} and such that $c(Y,\tilde\xi)\neq 0$. In the latter case we are done, so 
we may assume $\tilde Y\cong M_n$ for some $n\geq 1$. If $Y\cong\tilde Y$ we are done, 
therefore we assume $Y\not\cong\tilde Y$. This means that the third leg $L_3$ of the tree 
$\Ga$ has length greater than $1$, and with a little bit of Kirby calculus it is easy to see how  
this implies that $Y$ is orientation--preserving diffeomorphic to $S^3_\ga(T_{2,2n+1})$, 
the result of a rational surgery along the $(2,2n+1)$--torus knot, with $\ga\not\in\Z$, hence 
in particular with $\ga\neq 2n-1$. In this case the existence of a tight contact structure on $Y$ 
follows by~\cite[Theorem~1.1]{OSzI}, so the proof is finished. 
\end{proof}

\end{document}